\documentclass[12pt, reqno]{amsart}

\usepackage{amsmath, amsthm, amssymb}
\usepackage{enumerate}
\usepackage[margin=1.0in]{geometry}
\usepackage{xcolor}
\definecolor{cite}{rgb}{0.30,0.60,1.00}
\definecolor{url}{rgb}{0.00,0.00,0.80}
\definecolor{link}{rgb}{0.40,0.10,0.20}
\usepackage[pdfusetitle,colorlinks,linkcolor=link,urlcolor=url,citecolor=cite,pagebackref,breaklinks]{hyperref}
\usepackage{graphicx}
\usepackage{cleveref}
\usepackage{mathdots}
\usepackage{tikz-cd}
\usepackage{comment}
\usepackage{xypic}
\usepackage{mathtools}

\newtheorem{theorem}{Theorem}[section]

\newtheorem{proposition}[theorem]{Proposition}
\newtheorem{lemma}[theorem]{Lemma}

\newtheorem{corollary}[theorem]{Corollary}

\theoremstyle{definition}
\newtheorem{definition}[theorem]{Definition}
\theoremstyle{definition}
\newtheorem{remark}[theorem]{Remark}
\theoremstyle{definition}

\newcommand{\cComplex}{\mathbb{C}}
\newcommand{\multiplicativegroup}[1]{#1^{\ast}}
\newcommand{\characterGroup}[1]{\widehat{#1}}
\newcommand{\Hom}{\mathrm{Hom}}
\newcommand{\conjugate}[1]{\overline{#1}}
\newcommand{\sizeof}[1]{\left|#1\right|}
\newcommand{\finiteField}{\mathbb{F}}
\newcommand{\innerproduct}[2]{\left(#1,#2\right)}
\newcommand{\bilinearPairing}[2]{\left\langle#1,#2\right\rangle}
\newcommand{\fieldCharacter}{\psi}
\newcommand{\rsGammaFactor}[3][\fieldCharacter]{\gamma(#2 \times #3, #1)}
\newcommand{\gammaFactor}[3][\fieldCharacter]{\Gamma(#2 \times #3, #1)}
\newcommand{\normalizedGammaFactor}[3][\fieldCharacter]{\Gamma^{\ast}(#2 \times #3, #1)}
\newcommand{\centralCharacter}[1]{\omega_{#1}}
\newcommand{\ind}[3]{\operatorname{Ind}_{#1}^{#2}{#3}}
\newcommand{\whittaker}{\mathcal{W}}
\newcommand{\unipotentSubgroup}{Z}
\newcommand{\weyllong}{w}
\newcommand{\weylnmmatrix}[2]{\begin{pmatrix}
		& \identityMatrix{#1}\\
		\identityMatrix{#2}
\end{pmatrix}}

\newcommand{\transpose}[1]{{^t}{#1}}
\newcommand{\trace}{\mathrm{tr}}
\newcommand{\GL}{\mathrm{GL}}
\newcommand{\Mat}[2]{M_{#1 \times #2}}
\newcommand{\leviGroup}{D}
\newcommand{\parabolicGroup}{P}
\newcommand{\unipotentRadical}{N}
\newcommand{\weylnm}{\hat{w}}
\newcommand{\identityMatrix}[1]{I_{#1}}
\newcommand{\parabolicTensor}{\overline{\otimes}}

\newcommand{\repBesselFunction}[1]{\mathcal{J}_{#1,\fieldCharacter}}
\newcommand{\pairBesselFunction}[2]{\mathcal{J}_{#1,#2,\fieldCharacter}}
\newcommand{\grepBesselFunction}[2]{\mathcal{J}_{#1,#2}}

\newcommand{\whittakerVector}[1]{v_{#1,\fieldCharacter}}
\newcommand{\pairWhittakerVector}[2]{f_{\whittakerVector{#1},\whittakerVector{#2}}}

\newcommand{\spairWhittakerVector}[2]{v_{#1, #2, \fieldCharacter}}
\newcommand{\gspairWhittakerVector}[3]{v_{#1, #2, #3}}
\newcommand{\tripleWhittakerVector}[3]{v_{#1, #2, #3, \fieldCharacter}}
\newcommand{\whittakerFunctional}[1]{\ell_{#1,\fieldCharacter}}

\newcommand{\representationRes}[1]{\operatorname{Res}_{#1}}
\newcommand{\contragredient}[1]{{#1}^{\vee}}
\newcommand{\parabolicInduction}{\circ}
\newcommand{\mirabolic}{P}
\newcommand{\fourierTransform}[2]{\mathcal{F}_{#1}#2}
\newcommand{\diag}{\operatorname{diag}}
\newcommand{\identityMap}{\operatorname{id}}
\newcommand{\swapHomomorphism}{\operatorname{sw}}
\newcommand{\outerInvolution}[1]{#1^{\iota}}
\newcommand{\abs}[1]{\left|#1\right|}
\newcommand{\SchwartzSpacen}{\mathcal{S}\left(\finiteField^n\right)}
\newcommand{\FieldTrace}[1]{\operatorname{Tr}_{\finiteField_{#1} \slash \finiteField}}
\newcommand{\FieldNorm}[1]{\operatorname{N}_{\finiteField_{#1} \slash \finiteField}}

\usepackage{scalerel}
\usepackage{stackengine,wasysym}

\hypersetup{pdfauthor={David Soudry, Elad Zelingher},
	pdfsubject={Representation theory},
	pdfkeywords={Gamma factors, Finite fields, Gauss sums}}

\title[Gamma factors for representations of $\GL_n$]{On gamma factors for representations of finite general linear groups}

\author{David Soudry}
\address{School of Mathematical Sciences, Sackler Faculty of Exact Sciences, Tel-Aviv University, Israel 69978}
\email{soudry@tauex.tau.ac.il}

\author{Elad Zelingher}
\address{Department of Mathematics, University of Michigan, 1844 East Hall, 530 Church Street, Ann Arbor, MI 48109-1043 USA}
\email{eladz@umich.edu}

\begin{document}

\begin{abstract}
	We use the Langlands--Shahidi method in order to define the Shahidi gamma factor for a pair of irreducible generic representations of $\GL_n\left(\finiteField_q\right)$ and $\GL_m\left(\finiteField_q\right)$. We prove that the Shahidi gamma factor is multiplicative and show that it is related to the Jacquet--Piatetski-Shapiro--Shalika gamma factor. As an application, we prove a converse theorem based on the absolute value of the Shahidi gamma factor, and improve the converse theorem of Nien. As another application, we give explicit formulas for special values of the Bessel function of an irreducible generic representation of $\GL_n\left(\finiteField_q\right)$.
\end{abstract}

\maketitle

\section{Introduction}

In the representation theory of $p$-adic groups, one method of studying irreducible representations is by attaching local factors to the representations. These local factors are complex valued functions of a complex variable. They encode various properties of the representations in question. These local factors usually arise from global integrals representing $L$-functions attached to automorphic representations. Studying these local factors is crucial for understanding the global situation. This has been done successfully in many cases, including the pioneering works of Jacquet--Piatetski-Shapiro--Shalika \cite{Jacquet1983rankin} and Shahidi \cite{shahidi1984fourier, shahidi1990proof}.

Let $\finiteField$ be a finite field with cardinality $q$. A local theory of local factors often has a finite field analog. It allows one to attach ``local constants'' to irreducible representations of the $\finiteField$-points version of the group in consideration. One famous example is the book \cite{piatetskishapirobook} by Piatetski--Shapiro, in which he developed the theory of gamma factors for the tensor product representation of $\GL_2 \times \GL_1$ over finite fields. We also mention the works \cite{Roditty10, Nien14, YeZeligher18, liu2022gamma, liu2022converse} as examples. These local constants usually encode properties analogous to their local factors counterparts. Moreover, these local constant theories often allow one to consider ``toy models'' for analogous local problems. For instance, shortly after Nien's proof of the analog of Jacquet's conjecture for finite fields \cite{Nien14}, Chai proved the conjecture for the $p$-adic group case \cite{chai2019bessel}, where in his proof he used tools analogous to the ones used by Nien.

In her master's thesis \cite{Roditty10}, Roditty-Gershon defined a finite field analog of the gamma factor of Jacquet--Piatetski-Shapiro--Shalika \cite{Jacquet1983rankin}. This gamma factor represents the tensor product representation, attached to two irreducible generic representations $\pi$ and $\sigma$ of $\GL_n\left(\finiteField\right)$ and $\GL_m\left(\finiteField\right)$, respectively, and is denoted $\rsGammaFactor{\pi}{\sigma}$. Later, Rongqing Ye showed that $\rsGammaFactor{\pi}{\sigma}$ is related to its local field counterpart through level zero supercuspidal representations \cite{Ye18}. Using this relation and the local Langlands correspondence, Rongqing Ye and the second author were able to express $\rsGammaFactor{\pi}{\sigma}$ as a product of Gauss sums \cite{ye2021epsilon}.

The theory of the finite field version of the gamma factor associated to the tensor product, as it currently appears in the literature, is in some sense not complete. The first problem is that the gamma factor $\rsGammaFactor{\pi}{\sigma}$ is currently not defined for all irreducible generic representations $\pi$ and $\sigma$. It is only defined when $n \ge m$, and under the assumption that $\pi$ is cuspidal (and if $n = m$, $\sigma$ is also required to be cuspidal). One can tweak the proofs so they will work for all irreducible generic representations $\pi$ and $\sigma$, such that $\pi$ and $\contragredient{\sigma}$ have disjoint cuspidal support, but that is not enough in order to define $\rsGammaFactor{\pi}{\sigma}$ for all pairs $\pi$ and $\sigma$. One can try to define $\rsGammaFactor{\pi}{\sigma}$ naively using the expression involving the Bessel functions of $\pi$ and $\sigma$ (see \Cref{section:whittaker-models-and-bessel-functions} for the definition of the Bessel function), but this leads to the second problem. The second problem is that the current theory lacks the multiplicativity property of the gamma factor. If one naively extends the definition $\rsGammaFactor{\pi}{\sigma}$ using the approach suggested above, it is not clear that the gamma factor would be multiplicative. Both of these difficulties need to be resolved for applications as in \cite{zelingher2022values}.

The Langlands--Shahidi method provides an alternative approach that solves both of these problems. In this paper, we use this method to define a finite field version of the Shahidi gamma factor. We briefly describe the construction now. Let $\pi$ and $\sigma$ be representations of Whittaker type of $\GL_n\left(\finiteField\right)$ and $\GL_m\left(\finiteField\right)$, respectively. In \Cref{subsection:intertwining-operator}, we consider an intertwining operator $U_{\sigma, \pi} \colon \sigma \parabolicInduction \pi \rightarrow \pi \parabolicInduction \sigma$, where $\parabolicInduction$ denotes parabolic induction. In \Cref{subsection:shahidi-gamma-factor}, given Whittaker vectors $\whittakerVector{\pi} \in \pi$ and $\whittakerVector{\sigma} \in \sigma$, we define Whittaker vectors $\spairWhittakerVector{\pi}{\sigma} \in \pi \parabolicInduction \sigma$ and $\spairWhittakerVector{\sigma}{\pi} \in \sigma \parabolicInduction \pi$. By uniqueness of the Whittaker vectors, we have that there exists a constant $\gammaFactor{\pi}{\sigma} \in \cComplex$, such that $$U_{\sigma, \pi} \spairWhittakerVector{\sigma}{\pi} = \gammaFactor{\pi}{\sigma} \cdot \spairWhittakerVector{\pi}{\sigma}.$$
We call $\gammaFactor{\pi}{\sigma}$ the \emph{Shahidi gamma factor associated to $\pi$ and $\sigma$}. This is a finite analog of Shahidi's local coefficient \cite{shahidi1984fourier}.

We prove properties of $\gammaFactor{\pi}{\sigma}$, the most important one is that it is multiplicative (\Cref{thm:multiplicitavity-of-gamma-factors}).

\begin{theorem}
	Let $\pi$, $\sigma_1$ and $\sigma_2$ be representations of Whittaker type of $\GL_n\left(\finiteField\right)$, $\GL_{m_1}\left(\finiteField\right)$ and $\GL_{m_2}\left(\finiteField\right)$, respectively. Then $$ \gammaFactor{\pi}{(\sigma_1 \parabolicInduction \sigma_2)} = \gammaFactor{\pi}{\sigma_1} \cdot \gammaFactor{\pi}{\sigma_2}.$$
\end{theorem}

We also express $\gammaFactor{\pi}{\sigma}$ in terms of the Bessel functions associated with $\pi$ and $\sigma$ when both representations are irreducible. We show that if $n \ge m$, then up to some simple factors, $\gammaFactor{\pi}{\sigma}$ is given by the naive extension of $\rsGammaFactor{\pi}{\contragredient{\sigma}}$ discussed above (\Cref{thm:explicit-formula-for-gamma-factors-in-terms-of-bessel-functions}). We deduce a relation between the Shahidi gamma factor and the Jacquet--Piatetski-Shapiro--Shalika gamma factor (\Cref{cor:relation-between-intertwining-and-rankin-selberg}).
\begin{theorem}Let $\pi$ and $\sigma$ be irreducible generic representations of $\GL_n\left(\finiteField\right)$ and $\GL_m\left(\finiteField\right)$, respectively. Suppose that $\pi$ is cuspidal and $n \ge m$. If $n = m$, suppose that $\sigma$ is also cuspidal. Then
$$\gammaFactor{\pi}{\sigma} = q^{\frac{m \left(2n - m - 1\right)}{2}} \centralCharacter{\sigma}\left(-1\right) \rsGammaFactor{\pi}{\contragredient{\sigma}}.$$
\end{theorem}

The relation between both gamma factors allows us to give a representation theoretic interpretation for the absolute value of the Shahidi gamma factor. We show that, in some sense, the absolute value of the Shahidi gamma factor serves as a good substitute for the order of the pole of the local $L$-factor associated with the tensor product representation. Let us stress that the relation to the Jacquet--Piatetski-Shapiro--Shalika gamma factor is crucial for these results. The following theorem can be seen as an analog of \cite[Section 8.1]{Jacquet1983rankin}.

\begin{theorem}
	Let $\pi$ be an irreducible generic representation of $\GL_n\left(\finiteField\right)$ and let $\sigma$ be an irreducible cuspidal representation of $\GL_m\left(\finiteField\right)$. Then
	$$ \abs{q^{\frac{-nm}{2}} \cdot \gammaFactor{\pi}{\sigma}} = q^{- \frac{d_\pi\left(\sigma\right) m}{2}},$$
	where $d_\pi\left(\sigma\right)$ is the number of times $\sigma$ appears in the cuspidal support of $\pi$.
\end{theorem}
This allows us to deduce a converse theorem based on the absolute value of the normalized Shahidi gamma factor. Similar theorems in the local setting were given by Gan and his collaborators in many works (see \cite[Lemma 12.3]{gan2012representations}, \cite[Lemma A.6]{gan2016gross} and \cite[Lemma A.6]{atobe2017local}), but our proof is done on the ``group side'' rather than on the ``Galois side''.

\begin{theorem}
	Let $\pi_1$ and $\pi_2$ be two irreducible generic representations of $\GL_{n_1}\left(\finiteField\right)$ and $\GL_{n_2}\left(\finiteField\right)$, respectively. Assume that for every $m$ and every irreducible cuspidal representation $\sigma$ of $\GL_m\left(\finiteField\right)$ we have
	$$ \abs{q^{-\frac{n_1 m}{2}} \cdot \gammaFactor{\pi_1}{\sigma}} = \abs{q^{-\frac{n_2 m}{2}} \cdot \gammaFactor{\pi_2}{\sigma}}.$$
	Then $n_1 = n_2$ and $\pi_1 \cong \pi_2$.
\end{theorem}

Our results combined with Nien's converse theorem \cite{Nien14} allow us to deduce a converse theorem that holds under weaker assumptions. This is similar to \cite[Section 2.4]{jiang2015towards}.
\begin{theorem}
	Let $\pi_1$ and $\pi_2$ be irreducible generic representations of $\GL_n\left(\finiteField\right)$ with the same central character. Suppose that for any $1 \le m \le \frac{n}{2}$ and any irreducible cuspidal representation $\sigma$ of $\GL_m\left(\finiteField\right)$ we have
	$$ \gammaFactor{\pi_1}{\sigma} = \gammaFactor{\pi_2}{\sigma}.$$
	Then $\pi_1 \cong \pi_2$.
\end{theorem}

As another application of our results, we find explicit formulas for special values of the Bessel function of an irreducible generic representation $\pi$. The first formula (\Cref{thm:bessel-function-special-value}) expresses $\repBesselFunction{\pi}\left(\begin{smallmatrix}
	& \identityMatrix{n-1}\\
	c
\end{smallmatrix}\right)$ as an exotic Kloosterman sum \cite[Page 152]{katz1993estimates}. This formula is already known in the literature by the work of Curtis--Shinoda \cite{curtis2004zeta}, but our proof is based on multiplicativity of the Shahidi gamma factor, rather than on Deligne--Lusztig theory. The second formula we find (\Cref{thm:three-blocks-special-bessel-value}) expresses $\repBesselFunction{\pi}\left(\begin{smallmatrix}
& & -c'\\
& \identityMatrix{n-2}\\
c
\end{smallmatrix}\right)$ as a twisted convolution of values of the form $\repBesselFunction{\pi}\left(\begin{smallmatrix}
& \identityMatrix{n-1}\\
c
\end{smallmatrix}\right)$ and $\repBesselFunction{\pi}\left(\begin{smallmatrix}
& c'\\
\identityMatrix{n-1}
\end{smallmatrix}\right)$. Such a formula was given by Chang for $n=3$ \cite{chang1976decomposition} and then generalized by Shinoda--Tulunay \cite{shinoda2005representations} for $n=4$. Chang's method is based on the Gelfand--Graev algebra, while our method is based on formulas we found for the Shahidi gamma factor.

This paper is based on a unpublished note by the first author \cite{soudry1979} from 1979.

\section{Preliminaries}

\subsection{Parabolic induction}

Given a sequence of positive integers $n_1,\dots,n_r$, we denote by $\parabolicGroup_{n_1,\dots,n_r}$ the parabolic subgroup of $\GL_{n_1+\dots+n_r} \left(\finiteField\right)$ corresponding to the composition $\left(n_1,\dots,n_r\right)$. That is, $$\parabolicGroup_{n_1,\dots,n_r} = \leviGroup_{n_1,\dots,n_r} \ltimes \unipotentRadical_{n_1,\dots,n_r},$$ where
\begin{align*}
\leviGroup_{n_1,\dots,n_r} &= \left\{ \diag\left(g_1,\dots,g_r\right) \mid \forall 1 \le j \le r,  g_j \in \GL_{n_j}\left(\finiteField\right) \right\},\\
\unipotentRadical_{n_1,\dots,n_r} &= \left\{ \begin{pmatrix}
\identityMatrix{n_1} & \ast & \ast & \ast \\
& \identityMatrix{n_2} & \ast & \ast\\
& & \ddots & \ast \\	
& & &  \identityMatrix{n_r} 
\end{pmatrix} \right\}.
\end{align*}

Given representations $\pi_1,\dots,\pi_r$ of $\GL_{n_1}\left(\finiteField\right), \dots, \GL_{n_r}\left(\finiteField\right)$, respectively, we denote by $\pi_1 \parabolicTensor \dots \parabolicTensor \pi_r$ the inflation of $\pi_1 \otimes \dots \otimes \pi_r$ to $\parabolicGroup_{n_1,\dots,n_r}$. That is, $\pi_1 \parabolicTensor \dots \parabolicTensor \pi_r$ is a representation of $\parabolicGroup_{n_1,\dots,n_r}$, acting on the space of $\pi_1 \otimes \dots \otimes \pi_r$, and its action on pure tensors is given by $$\left(\pi_1 \parabolicTensor \dots \parabolicTensor \pi_r\right)\left(d u\right) v_1 \otimes \dots \otimes v_r = \pi_1\left(g_1\right)v_1 \otimes \dots \otimes \pi_r\left(g_r\right)v_r,$$
where $d = \diag\left(g_1,\dots,g_r\right) \in \leviGroup_{n_1,\dots,n_r}$ and $u \in \unipotentRadical_{n_1,\dots,n_r}$, and for every $1 \le j \le r$, $v_j \in \pi_j$.

The parabolic induction $\pi_1 \parabolicInduction \dots \parabolicInduction \pi_r$ is defined as the following representation of $\GL_{n_1+\dots+n_r}\left(\finiteField\right)$:

$$\pi_1 \parabolicInduction \dots \parabolicInduction \pi_r = \ind{\parabolicGroup_{n_1,\dots,n_r}}{\GL_{n_1+\dots+n_r}\left(\finiteField\right)}{\pi_1 \parabolicTensor \dots \parabolicTensor \pi_r}.$$

A representation $\pi$ of $\GL_n\left(\finiteField\right)$ is called \emph{cuspidal} if for every composition $\left(n_1,\dots,n_r\right) \ne \left(n\right)$ of $n$, there does not exist $0 \ne v \in \pi$ such that $v$ is invariant under $\unipotentRadical_{n_1,\dots,n_r}$, i.e., if $v \in \pi$ is such that $\pi\left(u\right)v = v$ for every $u \in \unipotentRadical_{n_1,\dots,n_r}$ then $v = 0$. If $\pi$ is irreducible, $\pi$ is cuspidal if and only if it is not a  subrepresentation of $\pi_1 \parabolicInduction \dots \parabolicInduction \pi_r$ for some $\pi_1, \dots, \pi_r$ as above, where $r > 1$.

By \cite[Theorem 2.4]{Gelfand70}, if $\pi$ is an irreducible representation of $\GL_n\left(\finiteField\right)$, then there exist $n_1,\dots,n_r > 0$ with $n_1 + \dots + n_r = n$ and irreducible cuspidal representations $\pi_1, \dots, \pi_r$, of $\GL_{n_1}\left(\finiteField\right), \dots, \GL_{n_r}\left(\finiteField\right)$, such that $\pi$ is isomorphic to a subrepresentation of the parabolic induction $\pi_1 \parabolicInduction \dots \parabolicInduction \pi_r$. Such $\pi_1, \dots, \pi_r$ are unique up to ordering. We define the \emph{cuspidal support of $\pi$} to be the multiset $\left\{ \pi_1, \dots, \pi_r \right\}$.

\subsection{Generic representations}

 Let $\fieldCharacter \colon \finiteField \rightarrow \multiplicativegroup{\cComplex}$ be a non-trivial additive character. Let $\unipotentSubgroup_n \le \GL_n\left(\finiteField\right)$ be the upper triangular unipotent subgroup. We define a character $\fieldCharacter \colon \unipotentSubgroup_n \rightarrow \multiplicativegroup{\cComplex}$ by the formula
$$ \fieldCharacter \begin{pmatrix}
1 & a_1 & \ast & \dots & \ast\\
& 1 & a_2 & \dots  & \ast \\
& & \ddots & \ddots  & \vdots \\
& & & 1 & a_{n-1} \\
& & & & 1
\end{pmatrix} = \psi\left( \sum_{k=1}^{n-1} a_k \right).$$

Let $\pi$ be a finite dimensional representation of $\GL_n\left(\finiteField\right)$. $\pi$ is said to be \emph{generic} if $$\Hom_{\unipotentSubgroup_n}\left(\representationRes{\unipotentSubgroup_n}\pi, \fieldCharacter \right) \ne 0.$$
This condition does not depend on the choice of $\fieldCharacter$. See \Cref{subsubsection:dependence-on-character}. We call a non-zero element in $\Hom_{\unipotentSubgroup_n}\left(\representationRes{\unipotentSubgroup_n}\pi, \fieldCharacter \right)$ a \emph{$\fieldCharacter$-Whittaker functional}. The representation $\pi$ is generic if and only if there exists $0 \ne v \in \pi$, such that $\pi \left(u
\right) v = \fieldCharacter\left(u\right) v$ for every $u \in \unipotentSubgroup_n$. We call such vector a \emph{Whittaker vector with respect to $\fieldCharacter$}, or a \emph{$\fieldCharacter$-Whittaker vector}. The dimension of the subspace spanned by the $\fieldCharacter$-Whittaker vectors of $\pi$ is $\dim \Hom_{\unipotentSubgroup_n}\left(\representationRes{\unipotentSubgroup_{n}} \pi, \fieldCharacter \right)$. 

\begin{definition}
	We say that \emph{$\pi$ is of Whittaker type} if $\pi$ is generic and the subspace spanned by its $\fieldCharacter$-Whittaker vectors is one-dimensional.
\end{definition}

By a well known result of Gelfand and Graev, we have that if $\pi$ is generic and irreducible, then it is of Whittaker type \cite[Theorem 0.5]{Gelfand70}, \cite[Corollary 5.6]{SilbergerZink00}. It is well known that irreducible cuspidal representations of $\GL_n\left(\finiteField\right)$ are generic \cite[Lemma 5.2]{SilbergerZink00}. The following result is also well known {\cite[Theorem 5.5]{SilbergerZink00}}.

\begin{theorem}\label{thm:unique-whittaker-vector-parabolic-induction}
	Let $\pi_1, \dots, \pi_r$ be representations of Whittaker type of $\GL_{n_1}\left(\finiteField\right), \dots, \GL_{n_r}\left(\finiteField\right)$, respectively. Then the parabolic induction $\pi_1 \parabolicInduction \dots \parabolicInduction \pi_r$ is a representation of Whittaker type.
\end{theorem}

\subsubsection{Whittaker models and Bessel functions}\label{section:whittaker-models-and-bessel-functions}

Let $\pi$ be an irreducible generic representation of $\GL_n\left(\finiteField\right)$. Since $\pi$ is of Whittaker type, Frobenius reciprocity implies that
$$\dim \Hom_{\GL_n\left(\finiteField\right)}\left(\pi, \ind{\unipotentSubgroup_n}{\GL_n\left(\finiteField\right)}{\fieldCharacter}\right) = 1.$$
We denote by $\whittaker\left(\pi, \fieldCharacter\right)$ the unique subspace of $\ind{\unipotentSubgroup_{n}}{\GL_n\left(\finiteField\right)}{\fieldCharacter}$ that is isomorphic to $\pi$. This is \emph{the Whittaker model of $\pi$ with respect to $\fieldCharacter$}.

Recall that for an irreducible representation $\pi$ of $\GL_n\left(\finiteField\right)$, we have that its contragredient $\contragredient{\pi}$ is isomorphic to $\outerInvolution{\pi}$, where $\outerInvolution{\pi}$ is the representation acting on the space of $\pi$ by $\outerInvolution{\pi}\left(g\right) = \pi\left(\outerInvolution{g}\right)$, where for $g \in \GL_n\left(\finiteField\right)$, $$\outerInvolution{g} = \transpose{\left(g^{-1}\right)}.$$ 
(This follows from the fact that for $g \in \GL_n\left(\finiteField\right)$, the trace characters of $\pi$ and $\contragredient{\pi}$ are related by $\trace \contragredient{\pi}\left(g\right) = \trace \pi (g^{-1})$, and from the fact that $g^{-1}$ and $\transpose{\left(g^{-1}\right)}$ are conjugate.)

Using the isomorphism $\contragredient{\pi} \cong \outerInvolution{\pi}$, we get an isomorphism of vector spaces $\whittaker\left(\pi,\fieldCharacter\right) \rightarrow \whittaker\left(\contragredient{\pi}, \fieldCharacter^{-1}\right)$, given by $W \mapsto \tilde{W}$, where
$$ \tilde{W}\left(g\right) = W\left(\weyllong_n \outerInvolution{g}\right),$$ and where $\weyllong_n \in \GL_n\left(\finiteField\right)$ is the long Weyl element $$\weyllong_n = \begin{pmatrix}
	& & & 1\\
	& & 1\\
	& \iddots\\
	1
\end{pmatrix}.$$

Under the realization of $\pi$ by its Whittaker  model $\whittaker\left(\pi, \fieldCharacter\right)$, the one-dimensional subspace spanned by the $\fieldCharacter$-Whittaker vectors of $\pi$ is realized as the one-dimensional subspace of $\whittaker\left(\pi,\fieldCharacter\right)$ consisting of functions $W \in \whittaker\left(\pi,\fieldCharacter\right)$, such that $W\left(gu\right) = \fieldCharacter\left(u\right) W\left(g\right)$, for every $u \in \unipotentSubgroup_n$ and every $g \in \GL_n\left(\finiteField\right)$. By \cite[Proposition 4.5]{Gelfand70}, there exists a (unique) element $W$ in this one-dimensional subspace such that $W\left(\identityMatrix{n}\right) = 1$. We call this $W$ the \emph{normalized Bessel function of $\pi$ with respect to $\fieldCharacter$}, and denote it by $\repBesselFunction{\pi}$. To summarize, the Bessel function $\repBesselFunction{\pi}$ is the unique element in $\whittaker\left(\pi,\fieldCharacter\right)$, such that
\begin{enumerate}
	\item $\repBesselFunction{\pi}\left(\identityMatrix{n}\right) = 1$.
	\item $\repBesselFunction{\pi}\left(g u\right) = \fieldCharacter\left(u\right) \repBesselFunction{\pi}\left(g
	\right)$, for every $g \in \GL_n\left(\finiteField\right)$ and $u \in \unipotentSubgroup_n$.
\end{enumerate}

The Bessel function enjoys the following identities that relate it to its complex conjugate and to its contragredient \cite[Propositions 2.15 and 3.5]{Nien14}.

\begin{proposition}\label{prop:contragredient-complex-conjugate}For any irreducible generic representation $\pi$ of $\GL_n\left(\finiteField\right)$ and any $g \in \GL_n\left(\finiteField\right)$, we have the following identities:
	\begin{enumerate}
		\item $\repBesselFunction{\pi}\left(g^{-1}\right) = \conjugate{\repBesselFunction{\pi}\left(g\right)}$.
		\item $\repBesselFunction{\pi}\left(g^{-1}\right) = {\grepBesselFunction{\contragredient{\pi}}{\fieldCharacter^{-1}}\left(g\right)}$.
	\end{enumerate}
\end{proposition}

\begin{remark}\label{rem:inner-product-and-whittaker-vector}
	Let $\whittakerVector{\pi}$ be a non-zero $\fieldCharacter$-Whittaker vector. If we choose an inner product $\innerproduct{\cdot}{\cdot}_\pi$ on $\pi$ which is invariant under the $\GL_n\left(\finiteField\right)$-action, we have that the assignment $\whittakerFunctional{\pi} \colon \pi \rightarrow \cComplex$ given by $v_\pi \mapsto \innerproduct{v_\pi}{\whittakerVector{\pi}}_\pi$ defines a Whittaker functional. The Whittaker model of $\pi$ can be described using Frobenius reciprocity as $\whittaker\left(\pi,\fieldCharacter\right) = \left\{ W_{v_\pi} \mid v_\pi \in \pi \right\}$, where for $g \in \GL_n\left(\finiteField\right)$ and $v_\pi \in \pi$, we define $W_{v_\pi}\left(g\right) = \innerproduct{\pi\left(g\right)v_\pi}{\whittakerVector{\pi}}_\pi$. The Bessel function is given by $$\repBesselFunction{\pi}\left(g\right) = \frac{\innerproduct{\pi\left(g\right)\whittakerVector{\pi}}{\whittakerVector{\pi}}_\pi}{\innerproduct{\whittakerVector{\pi}}{\whittakerVector{\pi}}_\pi}.$$
	All of the properties of the Bessel function listed above now follow immediately from the fact that $\innerproduct{\cdot}{\cdot}_\pi$ is an inner product, and that $\whittakerVector{\pi}$ is a $\fieldCharacter$-Whittaker vector. Moreover, the projection operator to the one-dimensional subspace spanned by the $\fieldCharacter$-Whittaker vectors $\operatorname{pr}_{\cComplex \whittakerVector{\pi}}$ can be described in two ways. The first way is by using the inner product, in which case for $v_\pi \in \pi$, $$\operatorname{pr}_{\cComplex \whittakerVector{\pi}}\left(v_\pi\right) = \frac{\innerproduct{v_\pi}{\whittakerVector{\pi}}_\pi}{\innerproduct{\whittakerVector{\pi}}{\whittakerVector{\pi}}_\pi} \whittakerVector{\pi}.$$ The second way is by averaging, in which case $$\operatorname{pr}_{\cComplex \whittakerVector{\pi}}\left(v_\pi\right) = \frac{1}{\sizeof{\unipotentSubgroup_n}} \sum_{u \in \unipotentSubgroup_n} \fieldCharacter^{-1}\left(u\right) \pi\left(u\right) v_\pi.$$ By completing $\whittakerVector{\pi}$ to an orthogonal basis of $\pi$ and using the fact that the subspace spanned by the $\fieldCharacter$-Whittaker vectors is one dimensional, we see that $$\trace \left(\operatorname{pr}_{\cComplex \whittakerVector{\pi}} \circ \pi\left(g\right)\right) = \repBesselFunction{\pi}\left(g\right).$$
	This is \cite[Proposition 4.5]{Gelfand70}.
\end{remark}

\subsection{Jacquet--Piatetski-Shapiro--Shalika gamma factors}\label{section:rankin-selberg-gamma-factors}

Let $\pi$ and $\sigma$ be irreducible generic representations of $\GL_n\left(\finiteField\right)$ and $\GL_m\left(\finiteField\right)$, respectively. For most $\pi$ and $\sigma$, one can define a constant attached to $\pi$ and $\sigma$ called the \emph{Jacquet--Piatetski-Shapiro--Shalika gamma factor of $\pi$ and $\sigma$}. It is also known as the \emph{Rankin--Selberg gamma factor of $\pi$ and $\sigma$}. This is a finite field analog of the definition given by Jacquet--Piatetski-Shapiro--Shalika \cite{Jacquet1983rankin} for $p$-adic groups. These were explained in Piatetski-Shapiro's lectures in 1976 and studied in an unpublished note from 1979 by the first author \cite{soudry1979}. The case $n > m$ was also studied in Roddity-Gershon's master's thesis under the supervision of the first author.

\subsubsection{The case $n > m$}

In her master's thesis \cite{Roditty10}, Edva Roditty-Gershon defined the Jacquet--Piatetski-Shapiro--Shalika gamma factor $\rsGammaFactor{\pi}{\sigma}$, under the assumption that $\pi$ is cuspidal and that $n > m$. Roddity-Gershon's thesis is unpublished, but her main results are presented by Nien in \cite{Nien14}. We briefly review these results now.

The first result is a functional equation that defines the Jacquet--Piatetski-Shapiro--Shalika gamma factor.
Suppose that $n > m$ and that $\pi$ is cuspidal. For any $W \in \whittaker\left(\pi,\fieldCharacter\right)$ and $W' \in \whittaker\left(\sigma,\fieldCharacter^{-1}\right)$, and any $0 \le  j \le n - m - 1$, we define
\begin{equation*}
	Z_j\left(W,W';\fieldCharacter\right) = \sum_{h \in \unipotentSubgroup_m \backslash \GL_m \left(\finiteField\right)} \sum_{x \in \Mat{\left(n-m-j-1\right)}{m}\left(\finiteField\right)} W\begin{pmatrix}
		h & &\\
		x & \identityMatrix{n-m-j-1}\\
		& & \identityMatrix{j+1}
	\end{pmatrix} W'\left(h\right).	\end{equation*}
We are now ready to state the functional equation.
\begin{theorem}[{\cite[Theorem 2.10]{Nien14}}]\label{thm:rankin-selberg-m-smaller-than-n-gamma-factor-definition}\label{thm:functional-equation-n-greater-than-m}
	 There exists a non-zero constant $\rsGammaFactor{\pi}{\sigma} \in \cComplex$, such that for every $0 \le j \le n - m - 1$, every $W \in \whittaker\left(\pi, \fieldCharacter\right)$ and every $W' \in \whittaker\left(\sigma, \fieldCharacter^{-1}\right)$, we have
	 
	 $$ q^{mj} \rsGammaFactor{\pi}{\sigma} Z_j\left(W, W' ; \fieldCharacter\right) = Z_{n-m-j-1} \left(\contragredient{\pi} \begin{pmatrix}
	 	\identityMatrix{m} &\\
	 	& \weyllong_{n-m}
	 \end{pmatrix} \tilde{W}, \tilde{W}'; \fieldCharacter^{-1} \right),$$
 where $\weyllong_{n-m} \in \GL_{n-m}\left(\finiteField\right)$ is the long Weyl element.
\end{theorem}
The second result expresses the gamma factor $\rsGammaFactor{\pi}{\sigma}$ in terms of the Bessel functions of $\pi$ and $\sigma$.
\begin{proposition}[{\cite[Proposition 2.16]{Nien14}}]\label{prop:rankin-selberg-gamma-factor-in-terms-of-bessel-functions}Under the assumptions above, we have
	$$\rsGammaFactor{\pi}{\sigma} = \sum_{h \in \unipotentSubgroup_m \backslash \GL_m\left(\finiteField\right)} \repBesselFunction{\pi}\begin{pmatrix}
		& \identityMatrix{n-m}\\
		h
	\end{pmatrix} \grepBesselFunction{\sigma}{\fieldCharacter^{-1}}\left(h\right).$$
\end{proposition}

It follows from \Cref{prop:rankin-selberg-gamma-factor-in-terms-of-bessel-functions} and \Cref{prop:contragredient-complex-conjugate} that $$\conjugate{\rsGammaFactor{\pi}{\sigma}} = \rsGammaFactor[\fieldCharacter^{-1}]{\contragredient{\pi}}{\contragredient{\sigma}}.$$

Moreover, applying \Cref{thm:rankin-selberg-m-smaller-than-n-gamma-factor-definition} twice, we get the following corollary regarding the absolute value of $\rsGammaFactor{\pi}{\sigma}$.

\begin{corollary}\label{prop:size-of-rs-gamma-factor-m-smaller-than-n}We have that
	 $$\rsGammaFactor{\pi}{\sigma}\rsGammaFactor[\fieldCharacter^{-1}]{\contragredient{\pi}}{\contragredient{\sigma}} = q^{-m \left(n - m - 1\right)},$$
	and therefore $$\abs{\rsGammaFactor{\pi}{\sigma}} = q^{-\frac{m \left(n - m - 1\right)}{2}}.$$
\end{corollary}

\subsubsection{The case $n = m$}
The case $n = m$ was discussed in Piatetski-Shapiro's lecture and is explained briefly in Rongqing Ye's work \cite{Ye18}.

Let $\SchwartzSpacen$ be the space of functions $\phi \colon \finiteField^n \rightarrow \cComplex$. For a function $\phi \in \SchwartzSpacen$, we define its Fourier transform $\fourierTransform{\fieldCharacter}{\phi} \colon \finiteField^n \rightarrow \cComplex$ by the formula $$\fourierTransform{\fieldCharacter}{\phi}\left(y\right) = \sum_{x \in \finiteField^n} \phi\left(x\right) \fieldCharacter\left(\bilinearPairing{x}{y}\right),$$ where if $x = \left(x_1,\dots,x_n\right) \in \finiteField^n$ and $y = \left(y_1,\dots,y_n\right) \in \finiteField^n$, then $\bilinearPairing{x}{y}$ is the standard pairing $$\bilinearPairing{x}{y} = \sum_{i=1}^n x_i y_i.$$

Let $\pi$ and $\sigma$ be irreducible cuspidal representations of $\GL_n\left(\finiteField\right)$. We define for any $W \in \whittaker\left(\pi,\fieldCharacter\right)$, $W' \in \whittaker\left(\sigma, \fieldCharacter^{-1}\right)$ and any $\phi \in \SchwartzSpacen$

$$Z\left(W,W',\phi;\fieldCharacter\right) = \sum_{g \in \unipotentSubgroup_n \backslash \GL_n\left(\finiteField\right)} W\left(g\right) W'\left(g\right) \phi\left(e_n g\right),$$
where $e_n = \left(0,\dots,0,1\right) \in \finiteField^n$. We are now ready to introduce the functional equation that defines $\rsGammaFactor{\pi}{\sigma}$.

\begin{theorem}[{\cite[Theorem 2.3]{Ye18}}]\label{thm:rankin-selberg-m-equals-n-gamma-factor-definition}
	There exists a non-zero constant $\rsGammaFactor{\pi}{\sigma}$, such that for any $W \in \whittaker\left(\pi,\fieldCharacter\right)$, $W' \in \whittaker\left(\sigma, \fieldCharacter^{-1}\right)$, and any $\phi \in \SchwartzSpacen$ with $\phi\left(0\right) = 0$, we have
	$$ Z(\tilde{W}, \tilde{W}', \fourierTransform{\fieldCharacter}{\phi}; \fieldCharacter^{-1}) = \rsGammaFactor{\pi}{\sigma} Z\left(W,W',\phi;\fieldCharacter\right).$$
\end{theorem}

Similarly to the case $n > m$, we have an expression of $\rsGammaFactor{\pi}{\sigma}$ in terms of the Bessel functions of $\pi$ and $\sigma$.

\begin{proposition}[{\cite[Equation (4.4)]{Ye18}}]\label{prop:rankin-selberg-gamma-factor-m-equals-n-in-terms-of-bessel-functions}
	Let $\pi$ and $\sigma$ be irreducible cuspidal representations of $\GL_n\left(\finiteField\right)$. Then
	$$ \rsGammaFactor{\pi}{\sigma} = \sum_{g \in \unipotentSubgroup_n \backslash \GL_n\left(\finiteField\right)} \repBesselFunction{\pi}\left(g\right) \grepBesselFunction{\sigma}{\fieldCharacter^{-1}}\left(g\right) \fieldCharacter\left( \bilinearPairing{e_n g^{-1}}{e_1} \right),$$ where $e_1 = \left(1,0,\dots,0\right) \in \finiteField^n$.
\end{proposition}

It follows from \Cref{prop:rankin-selberg-gamma-factor-m-equals-n-in-terms-of-bessel-functions} and \Cref{prop:contragredient-complex-conjugate} that $$\rsGammaFactor[\fieldCharacter^{-1}]{\contragredient{\pi}}{\contragredient{\sigma}} = \conjugate{\rsGammaFactor{\pi}{\sigma}}.$$

We now move to discuss the absolute value of $\rsGammaFactor{\pi}{\sigma}$. In order to do that, we first explain how to extend the functional equation in \Cref{thm:rankin-selberg-m-equals-n-gamma-factor-definition} to all functions in $\SchwartzSpacen$ for most cases. To begin, we notice that for the indicator function of $0 \in \finiteField^n$, which we denote $\delta_0$, we have that $Z\left(W, W', \delta_0; \fieldCharacter\right) = 0$.
We also notice that if $\pi$ is not isomorphic to $\contragredient{\sigma}$, then $Z\left(W, W' , 1; \fieldCharacter\right) = 0$, where $1$ represents the constant function. This is because
$$Z\left(W, W' , 1; \fieldCharacter\right) = \sum_{g \in \unipotentSubgroup_n \backslash \GL_n\left(\finiteField\right)} W\left(g\right) W'\left(g\right)$$ defines a $\GL_n\left(\finiteField\right)$-invariant pairing $\whittaker\left(\pi,\fieldCharacter\right) \otimes \whittaker\left(\sigma, \fieldCharacter^{-1}\right) \rightarrow \cComplex$, but such non-trivial pairing exists only when $\pi$ is isomorphic to $\contragredient{\sigma}$. These two observations imply the following extension of the functional equation, in the special case where $\pi$ is not isomorphic to $\contragredient{\sigma}$.
\begin{proposition}\label{prop:extended-functional-equation-n-equals-m}
	Suppose that $\pi \ncong \contragredient{\sigma}$. Then for any $\phi \in \SchwartzSpacen$ we have
	$$Z(\tilde{W}, \tilde{W}', \fourierTransform{\fieldCharacter}{\phi}; \fieldCharacter^{-1}) = \rsGammaFactor{\pi}{\sigma} Z\left(W,W',\phi;\fieldCharacter\right).$$
\end{proposition}
\begin{proof}
	Write $\phi = \phi_0 + \phi_1$, where $\phi_0 = \phi - \phi(0)$ and $\phi_1 = \phi(0)$. Then $\phi_0\left(0\right) = 0$ and $\fourierTransform{\fieldCharacter}{\phi_1} = q^n \phi\left(0\right) \delta_0$. Since $Z$ is linear in $\phi$, we have from the discussion above that $$Z\left(W,W',\phi;\fieldCharacter\right) = Z\left(W,W',\phi_0;\fieldCharacter\right) $$ and that $$Z(\tilde{W},\tilde{W}',\fourierTransform{\fieldCharacter}{\phi};\fieldCharacter^{-1}) = Z(\tilde{W},\tilde{W}',\fourierTransform{\fieldCharacter}{\phi_0};\fieldCharacter^{-1}).$$ The statement now follows from \Cref{thm:rankin-selberg-m-equals-n-gamma-factor-definition}.
\end{proof}

As a result, we get the following corollary regarding the absolute value of $\rsGammaFactor{\pi}{\sigma}$.

\begin{corollary}\label{cor:size-of-rankin-selberg-m-equals-n}Let $\pi$ and $\sigma$ be irreducible cuspidal representations of $\GL_n\left(\finiteField\right)$ such that $\pi \ncong \contragredient{\sigma}$. Then
	$$\rsGammaFactor{\pi}{\sigma} \rsGammaFactor[\fieldCharacter^{-1}]{\contragredient{\pi}}{\contragredient{\sigma}} = q^{n},$$
	and therefore
	$$ \abs{\rsGammaFactor{\pi}{\sigma}} = q^{\frac{n}{2}}.$$
\end{corollary}
\begin{proof}
	This follows by applying \Cref{prop:extended-functional-equation-n-equals-m} twice, and from the fact that the Fourier transform satisfies
	$$  \fourierTransform{\fieldCharacter^{-1}}{\fourierTransform{\fieldCharacter}{\phi}} = q^n \phi,$$ for any $\phi \in \SchwartzSpacen$.
\end{proof}

We are left to deal with the case $\pi \cong \contragredient{\sigma}$. In this case, the gamma factor $\rsGammaFactor{\pi}{\contragredient{\pi}}$ can be computed explicitly and it equals $-1$, see \Cref{appendix:computation-of-rankin-selberg-pi-pi-dual}.

We summarize all cases in the following proposition.

\begin{proposition}\label{prop:size-of-rs-gamma-factor-m-equals-n}
	Let $\pi$ and $\sigma$ be irreducible cuspidal representations of $\GL_n\left(\finiteField\right)$.
	\begin{itemize}
		\item If $\pi \ncong \contragredient{\sigma}$ then $\abs{\rsGammaFactor{\pi}{\sigma}} = q^{\frac{n}{2}}$.
		\item If $\pi \cong \contragredient{\sigma}$ then $\abs{\rsGammaFactor{\pi}{\sigma}} = 1$.
	\end{itemize}
\end{proposition}

\section{Shahidi gamma factors (local coefficients)}

In this section, we use the Langlands--Shahidi method in order to define a gamma factor for two representations of Whittaker type of finite general linear groups. This is the finite field analog of Shahidi's local coefficient, which uses an intertwining operator. The treatment in Sections \ref{subsection:intertwining-operator}-\ref{subsection:multiplicativity-of-gamma-factors} is a finite field analog of Shahidi's work on local coefficients over local fields \cite{shahidi1984fourier}. Unlike the Jacquet--Piatetski-Shapiro--Shalika gamma factors discussed in \Cref{section:rankin-selberg-gamma-factors}, the Shahidi gamma factor can be defined uniformly for all irreducible generic representations of $\GL_n\left(\finiteField\right)$ and $\GL_m\left(\finiteField\right)$, regardless of $n > m$ or whether the representations are cuspidal. We prove properties of the Shahidi gamma factor, where the most important one is the multiplicativity property, which explains how this gamma factor behaves under parabolic induction. We end this section by expressing the Shahidi gamma factor in terms of the Bessel functions associated with the representations, and showing its relation to the Jacquet--Piatetski-Shapiro--Shalika gamma factor.

\subsection{The intertwining operator}\label{subsection:intertwining-operator}

Let $n$ and $m$ be positive integers and let $\pi$ and $\sigma$ be representations of $\GL_n\left(\finiteField\right)$ and $\GL_m\left(\finiteField\right)$, respectively. We define a linear map $ \sigma \otimes \pi \rightarrow \pi \otimes \sigma$ acting on pure tensors by component swap:
$$ \swapHomomorphism_{\sigma,\pi} \left(v_\sigma \otimes v_\pi\right) = v_\pi \otimes v_\sigma.$$ For a function $f \colon \GL_{n + m}\left(\finiteField\right) \rightarrow \sigma \otimes \pi$, we denote by $\tilde{f} \colon \GL_{n + m}\left(\finiteField\right) \rightarrow \pi \otimes \sigma$ the function $\tilde{f}\left(g\right) = \swapHomomorphism_{\sigma,\pi}\left(f(g)\right)$.

We consider the following intertwining operator $T_{\sigma,\pi} \colon \sigma \parabolicInduction \pi \rightarrow \pi \parabolicInduction \sigma$, defined for $f \in \sigma \parabolicInduction \pi$ and $g \in \GL_{n+m}\left(\finiteField\right)$ by the formula 
$$T_{\sigma,\pi} f\left(g\right) = \sum_{p \in \parabolicGroup_{n,m} } {\left(\pi \parabolicTensor \sigma\right)\left(p^{-1}\right) {\tilde{f}\left(\weylnm_{n,m} p g\right)}},$$ where $\weylnm_{n,m}$ is the following Weyl element $$\weylnm_{n,m} = \weylnmmatrix{m}{n}.$$

Using the decomposition $\parabolicGroup_{n,m} = \leviGroup_{n,m} \ltimes \unipotentRadical_{n,m}$, we may write every $p \in \parabolicGroup_{n,m}$ in a unique way $p = du$, where $u \in \unipotentRadical_{n,m}$ and $d = \diag\left(g_1, g_2\right) \in \leviGroup_{n,m}$, with  $g_1 \in \GL_n\left(\finiteField\right)$ and $g_2 \in \GL_m\left(\finiteField\right)$. It follows from the identity $$\diag\left(g_2,g_1\right) \weylnm_{n,m} = \weylnm_{n,m} \diag\left(g_1,g_2\right),$$ and from the left $\leviGroup_{m,n}$-equivariance property of $f$ that
$$T_{\sigma,\pi}f\left(g\right) = \sizeof{\leviGroup_{n,m}} \cdot U_{\sigma,\pi}f\left(g\right),$$
where
$$U_{\sigma,\pi}f\left(g\right) = \sum_{u \in \unipotentRadical_{n,m}} \tilde{f}\left(\weylnm_{n,m} u g\right).$$

By construction, we have that $T_{\sigma, \pi}$ and $U_{\sigma, \pi}$ are non-zero elements of the space $$\Hom_{\GL_{n + m}\left(\finiteField\right)}\left( \sigma \parabolicInduction \pi, \pi \parabolicInduction \sigma \right).$$

\subsection{The Shahidi gamma factor}\label{subsection:shahidi-gamma-factor}

Suppose now that $\pi$ and $\sigma$ are representations of Whittaker type of $\GL_n\left(\finiteField\right)$ and $\GL_m\left(\finiteField\right)$, respectively. By \Cref{thm:unique-whittaker-vector-parabolic-induction} we have that the parabolically induced representations $\sigma \parabolicInduction \pi$ and $\pi \parabolicInduction \sigma$ are also of Whittaker type. Let $\whittakerVector{\sigma} \in \sigma$ and $\whittakerVector{\pi} \in \pi$ be non-zero $\fieldCharacter$-Whittaker vectors for $\sigma$ and $\pi$, respectively. We may define a non-zero $\fieldCharacter$-Whittaker vector $\pairWhittakerVector{\sigma}{\pi}$ for $\sigma \parabolicInduction \pi$ by the formula $$\pairWhittakerVector{\sigma}{\pi} \left(g\right) = \begin{dcases}
\psi\left(u\right) \left(\sigma \parabolicTensor \pi\right) \left(p\right) \whittakerVector{\sigma} \otimes \whittakerVector{\pi} & g = p \weylnm_{n,m} u,\, p \in \parabolicGroup_{m,n}, u \in \unipotentSubgroup_{n + m}, \\
0 & \text{otherwise}.
\end{dcases} $$
Similarly, we may define $\pairWhittakerVector{\pi}{\sigma} \in \pi \parabolicInduction \sigma$.

Since $U_{\sigma, \pi}$ is an intertwining operator, we have that $U_{\sigma, \pi} \pairWhittakerVector{\sigma}{\pi}$ is a $\fieldCharacter$-Whittaker vector of $\pi \parabolicInduction \sigma$. Since $\pairWhittakerVector{\pi}{\sigma}$ is the unique non-zero $\fieldCharacter$-Whittaker vector of $\pi \parabolicInduction \sigma$ up to scalar, we must have that
$$U_{\sigma, \pi} \pairWhittakerVector{\sigma}{\pi} = \gamma \cdot \pairWhittakerVector{\pi}{\sigma},$$ where $\gamma \in \cComplex$. It is easy to check that this number $\gamma$ does not depend on the choice of $\fieldCharacter$-Whittaker vectors $\whittakerVector{\sigma}$ and $\whittakerVector{\pi}$.

In order to ease the notation, we denote $\spairWhittakerVector{\sigma}{\pi} = \pairWhittakerVector{\sigma}{\pi}$, where we suppress $\whittakerVector{\sigma}$ and $\whittakerVector{\pi}$ from the notation. Similarly, we denote $\spairWhittakerVector{\pi}{\sigma} = \pairWhittakerVector{\pi}{\sigma}$.

\begin{definition}
	The Shahidi gamma factor of $\pi$ and $\sigma$ with respect to $\fieldCharacter$ is the unique number $\gammaFactor{\pi}{\sigma} \in \cComplex$, such that $$U_{\sigma, \pi} \spairWhittakerVector{\sigma}{\pi} = \gammaFactor{\pi}{\sigma} \cdot \spairWhittakerVector{\pi}{\sigma}.$$
\end{definition}

\begin{remark}
	If $\pi \parabolicInduction \sigma$ is irreducible, then so is $\sigma \parabolicInduction \pi$, and since $U_{\sigma, \pi}$ is a non-zero intertwining operator, it is an isomorphism and $\gammaFactor{\pi}{\sigma}$ must be non-zero. However, in the general case it is not obvious at this point that $\gammaFactor{\pi}{\sigma}$ is non-zero. We will show this later.
\end{remark} 

\begin{remark}
	As in \Cref{rem:inner-product-and-whittaker-vector}, we may choose invariant inner products $\innerproduct{\cdot}{\cdot}_\pi$ and $\innerproduct{\cdot}{\cdot}_\sigma$ on $\pi$ and $\sigma$, respectively. We then have a natural inner product $\innerproduct{\cdot}{\cdot}_{\sigma \otimes \pi}$ on $\sigma \otimes \pi$, which defines an inner product on $\sigma \parabolicInduction \pi$ by the formula
	$$ \innerproduct{f_1}{f_2}_{\sigma \parabolicInduction \pi} = \sum_{g \in \parabolicGroup_{m,n} \backslash \GL_{n+m}\left(\finiteField\right)} \innerproduct{f_1\left(g\right)}{f_2\left(g\right)}_{\sigma \otimes \pi}.$$
	Using this inner product, the Whittaker functional $\whittakerFunctional{\sigma \parabolicInduction \pi} \left(f\right) = \innerproduct{f}{\spairWhittakerVector{\sigma}{\pi}}_{\sigma \parabolicInduction \pi}$ is related to the Whittaker functionals $\whittakerFunctional{\sigma}\left(v_\sigma\right) = \innerproduct{v_\sigma}{\whittakerVector{\sigma}}_\sigma$ and $\whittakerFunctional{\pi}\left(v_\pi\right) = \innerproduct{v_\pi}{\whittakerVector{\pi}}_\pi$ by the formula
	$$ \whittakerFunctional{\sigma \parabolicInduction \pi}\left(f\right) = \sum_{u \in \unipotentRadical_{n,m}} \whittakerFunctional{\sigma} \otimes \whittakerFunctional{\pi} \left( f\left(\weylnm_{n,m} u\right) \right) \fieldCharacter^{-1}\left(u\right).$$
	Similarly, by exchanging the roles of $\pi$ and $\sigma$, we have that Whittaker functional $\whittakerFunctional{\pi \parabolicInduction \sigma}$ is given by a similar formula. Using the definitions of the inner products, and the fact that elements in $\pi \parabolicInduction \sigma$ are left invariant under $\unipotentRadical_{n,m}$, we see that $U_{\pi, \sigma}$ is the adjoint of $U_{\sigma, \pi}$, with respect to our choice of inner products. Using the relation between $\spairWhittakerVector{\sigma}{\pi}$ and $\spairWhittakerVector{\pi}{\sigma}$, we obtain the following relation
	$$\whittakerFunctional{\sigma \parabolicInduction \pi} \circ U_{\pi, \sigma} = \gammaFactor{\pi}{\sigma} \cdot \whittakerFunctional{\pi \parabolicInduction \sigma}.$$
	This is how the Shahidi gamma factor is usually defined in the literature.
\end{remark}

\subsubsection{Dependence on $\fieldCharacter$}\label{subsubsection:dependence-on-character}
For any $a \in \multiplicativegroup{\finiteField}$, let $\fieldCharacter_a \colon \finiteField \rightarrow \multiplicativegroup{\cComplex}$ be the additive character $$\fieldCharacter_a\left(x\right) = \fieldCharacter\left(ax\right).$$ It is well known that all non-trivial additive characters of $\finiteField$ are of the form $\fieldCharacter_a$ for some $a \in \multiplicativegroup{\finiteField}$. In this section, we give a relation between $\gammaFactor{\pi}{\sigma}$ and $\gammaFactor[\fieldCharacter_a]{\pi}{\sigma}$.

Let $a \in \multiplicativegroup{\finiteField}$. Suppose that $\tau$ is a generic representation of $\GL_k\left(\finiteField\right)$ with a non-zero  $\fieldCharacter$-Whittaker vector $\whittakerVector{\tau}$. Let $$d_{k} = \diag\left(1,a,a^2,\dots,a^{k-1}\right).$$ Then we have that $\tau\left(d_{k}\right) \whittakerVector{\tau}$ is a non-zero $\fieldCharacter_a$-Whittaker vector of $\tau$. The map $v \mapsto \tau\left(d_{k}\right) v$ is a linear isomorphism from the subspace spanned by the $\fieldCharacter$-Whittaker vectors of $\tau$ to the subspace spanned by the $\fieldCharacter_a$-Whittaker vectors of $\tau$. In particular, if $\whittakerVector{\tau}$ is the unique (up to scalar multiplication) $\fieldCharacter$-Whittaker vector of $\tau$, then $\tau\left(d_{k}\right) \whittakerVector{\tau}$ is the unique (up to scalar multiplication) $\fieldCharacter_a$-Whittaker vector of $\tau$.

Let $\pi$ and $\sigma$ be representations of Whittaker type of $\GL_n\left(\finiteField\right)$ and $\GL_m\left(\finiteField\right)$, respectively. Let $\whittakerVector{\pi} \in \pi$ and $\whittakerVector{\sigma} \in \sigma$ be non-zero $\fieldCharacter$-Whittaker vectors. Assume that $\pi$ and $\sigma$ have central characters, and denote them by $\centralCharacter{\pi}$ and $\centralCharacter{\sigma}$, respectively. Let $$\gspairWhittakerVector{\sigma}{\pi}{\fieldCharacter_a} = f_{\sigma\left(d_m\right) \whittakerVector{\sigma}, \pi\left(d_n\right)\whittakerVector{\pi}}.$$ Similarly, we define $\gspairWhittakerVector{\pi}{\sigma}{\fieldCharacter_a}$.

We first express $\gspairWhittakerVector{\sigma}{\pi}{\fieldCharacter_a}$ in terms of $\spairWhittakerVector{\sigma}{\pi}$. We will use this relation later to show a relation between the gamma factors $\gammaFactor[\fieldCharacter_a]{\pi}{\sigma}$ and $\gammaFactor{\pi}{\sigma}$.

\begin{proposition}\label{prop:identity-of-whittaker-vector-for-different-character}
	 We have $$\gspairWhittakerVector{\sigma}{\pi}{\fieldCharacter_a} =\centralCharacter{\sigma}\left(a\right)^{-n} \rho\left(d_{n+m}\right) \spairWhittakerVector{\sigma}{\pi},$$ where $\rho\left(d_{n+m}\right)$ denotes right translation by $d_{n+m}$.
\end{proposition}
\begin{proof}
	Let $f = \centralCharacter{\sigma}\left(a\right)^{-n} \rho\left(d_{n+m}\right) \spairWhittakerVector{\sigma}{\pi} \in \sigma \parabolicInduction \pi$. 
	By the discussion above, $f$ is a $\fieldCharacter_a$-Whittaker vector of $\sigma \parabolicInduction \pi$. 
	
	We have that $$f\left( \weylnm_{n,m} \right) = \centralCharacter{\sigma}\left(a\right)^{-n} \spairWhittakerVector{\sigma}{\pi}\left(\weylnm_{n,m}  d_{n+m} \right).$$
	Writing $d_{n + m} = \diag\left(d_n, a^n d_m\right)$, we have $\weylnm_{n,m} d_{n+m} = \diag\left(a^n d_m, d_n\right) \weylnm_{n,m}$, and hence
	$$f\left( \weylnm_{n,m}  \right) = \left(\sigma\left(d_m\right) \otimes \pi\left(d_n\right) \right) \spairWhittakerVector{\sigma}{\pi}\left(\weylnm_{n,m} \right) = \sigma\left(d_m\right) \whittakerVector{\sigma} \otimes \pi\left(d_n\right) \whittakerVector{\pi}.$$
	This shows that $f = \gspairWhittakerVector{\sigma}{\pi}{\fieldCharacter_a}$, as both are $\fieldCharacter_a$-Whittaker vectors in $\sigma \parabolicInduction \pi$, and both agree at the point $\weylnm_{n,m}$.
\end{proof}

\begin{theorem}\label{thm:change-of-additive-character}We have
	$$ \gammaFactor[\fieldCharacter_a]{\pi}{\sigma} = \centralCharacter{\pi}\left(a\right)^m \cdot \centralCharacter{\sigma}\left(a\right)^{-n} \cdot \gammaFactor{\pi}{\sigma}.$$
\end{theorem}
\begin{proof}
	By definition, we have that $$\gammaFactor[\fieldCharacter_a]{\pi}{\sigma} \gspairWhittakerVector{\pi}{\sigma}{\fieldCharacter_a} = U_{\sigma, \pi}  \gspairWhittakerVector{\sigma}{\pi}{\fieldCharacter_a}.$$	
	By \Cref{prop:identity-of-whittaker-vector-for-different-character}, $$\gammaFactor[\fieldCharacter_a]{\pi}{\sigma}  \centralCharacter{\pi}\left(a\right)^{-m}  \rho\left(d_{n+m}\right) \spairWhittakerVector{\pi}{\sigma} = \centralCharacter{\sigma}\left(a\right)^{-n}  U_{\sigma, \pi}  \rho\left(d_{n+m}\right) \spairWhittakerVector{\sigma}{\pi}.$$ 

	Therefore, we get that
	$$\gammaFactor[\fieldCharacter_a]{\pi}{\sigma} \centralCharacter{\sigma}\left(a\right)^{n} \centralCharacter{\pi}\left(a\right)^{-m} \spairWhittakerVector{\pi}{\sigma} =   U_{\sigma, \pi}  \spairWhittakerVector{\sigma}{\pi},$$
	which implies that $$\gammaFactor[\fieldCharacter_a]{\pi}{\sigma} \centralCharacter{\sigma}\left(a\right)^{n} \centralCharacter{\pi}\left(a\right)^{-m} = \gammaFactor{\pi}{\sigma},$$
	as required.
\end{proof}

\subsubsection{Relation between $\gammaFactor{\pi}{\sigma}$ and $\gammaFactor[{\fieldCharacter^{-1}}]{\contragredient{\sigma}}{\contragredient{\pi}}$}

In this section, we analyze the relation between $\gammaFactor{\pi}{\sigma}$ and $\gammaFactor{\contragredient{\sigma}}{\contragredient{\pi}}$.

Recall that for a finite dimensional representation $\tau$ of $\GL_k\left(\finiteField\right)$, we have that $\contragredient{\tau} \cong \outerInvolution{\tau}$. See \Cref{section:whittaker-models-and-bessel-functions}. If $\whittakerVector{\tau}$ is a non-zero $\fieldCharacter$-Whittaker vector for $\tau$, then $\tau\left(\weyllong_k\right) \whittakerVector{\tau}$ is a non-zero $\fieldCharacter^{-1}$-Whittaker vector for $\outerInvolution{\tau}$.

We have that $$\outerInvolution{\pi} \parabolicInduction \outerInvolution{\sigma} \cong \outerInvolution{\left(\sigma \parabolicInduction \pi\right)}$$ by the isomorphism $S_{\sigma,\pi} \colon \outerInvolution{\left(\sigma \parabolicInduction \pi\right)} \cong \outerInvolution{\pi} \parabolicInduction \outerInvolution{\sigma}$ that sends $f \in \outerInvolution{\left(\sigma \parabolicInduction \pi\right)}$ to the function
$$\left(S_{\sigma,\pi} f\right)\left(g\right) = \tilde{f}\left(\weylnm_{n,m} \outerInvolution{g} \right).$$ 

	Let $$\gspairWhittakerVector{\outerInvolution{\pi}}{\outerInvolution{\sigma}}{\fieldCharacter^{-1}} = f_{ \pi\left(\weyllong_n\right) \whittakerVector{\pi}, \sigma\left(\weyllong_m\right) \whittakerVector{\sigma}} \in \outerInvolution{\pi} \parabolicInduction \outerInvolution{\sigma}.$$
Then $\gspairWhittakerVector{\outerInvolution{\pi}}{\outerInvolution{\sigma}}{\fieldCharacter^{-1}}$ is a non-zero $\fieldCharacter^{-1}$-Whittaker vector of $\outerInvolution{\pi} \parabolicInduction \outerInvolution{\sigma}$. On the other hand, by the discussion above, a non-zero $\fieldCharacter^{-1}$-Whittaker vector of $\outerInvolution{\left(\sigma \parabolicInduction \pi\right)}$ is given by $\rho\left(\weyllong_{m+n}\right) \spairWhittakerVector{\sigma}{\pi}$, where $\rho\left(\weyllong_{m+n}\right)$ represents right translation by $\weyllong_{m+n}$. Therefore $S_{\sigma,\pi} \rho\left(\weyllong_{m+n}\right) \spairWhittakerVector{\sigma}{\pi}$ is another non-zero $\fieldCharacter^{-1}$-Whittaker vector of $\outerInvolution{\pi} \parabolicInduction \outerInvolution{\sigma}$.

\begin{proposition}We have\label{prop:pi-and-sigma-involution-whittaker-vector}
	\begin{equation}\label{eq:pi-and-sigma-involution-whittaker-vector}
		\gspairWhittakerVector{\outerInvolution{\pi}}{\outerInvolution{\sigma}}{\fieldCharacter^{-1}} = S_{\sigma,\pi} \rho\left( \weyllong_{m+n} \right) \spairWhittakerVector{\sigma}{\pi}. 
	\end{equation}
\end{proposition}
\begin{proof}
	We have that $$S_{\sigma,\pi} \rho\left( \weyllong_{m+n} \right) \spairWhittakerVector{\sigma}{\pi}\left( \weylnm_{m,n} \right) = \swapHomomorphism_{\sigma,\pi}{\spairWhittakerVector{\sigma}{\pi}}\left(\weyllong_{m+n} \right) = \pi\left(\weyllong_n\right) \whittakerVector{\pi} \otimes \sigma\left(\weyllong_m\right) \whittakerVector{\sigma},$$
	where in the last step we used the fact that $\diag\left(\weyllong_m, \weyllong_n\right) \weylnm_{n,m} = \weyllong_{n+m}$.
	
	Since $S_{\sigma,\pi} \rho\left( \weyllong_{m+n} \right) \spairWhittakerVector{\sigma}{\pi}$ and $\gspairWhittakerVector{\outerInvolution{\pi}}{\outerInvolution{\sigma}}{\fieldCharacter^{-1}}$ are both $\fieldCharacter^{-1}$-Whittaker vectors for the representation of Whittaker type $\outerInvolution{\pi} \parabolicInduction \outerInvolution{\sigma}$, and they both agree at the point $\weylnm_{m,n}$, they are equal.
\end{proof}

Similarly, let $$\gspairWhittakerVector{\outerInvolution{\sigma}}{\outerInvolution{\pi}}{\fieldCharacter^{-1}} = f_{ \sigma\left(\weyllong_m\right) \whittakerVector{\sigma}, \pi\left(\weyllong_n\right) \whittakerVector{\pi}} \in \outerInvolution{\sigma} \parabolicInduction \outerInvolution{\pi}.$$ Using \Cref{prop:pi-and-sigma-involution-whittaker-vector} with roles of the representations $\pi$ and $\sigma$ exchanged, we have
\begin{equation}\label{eq:sigma-and-pi-involution-whittaker-vector}
	\gspairWhittakerVector{\outerInvolution{\sigma}}{\outerInvolution{\pi}}{\fieldCharacter^{-1}} = S_{\pi,\sigma} \rho\left( \weyllong_{m+n} \right) \spairWhittakerVector{\pi}{\sigma}.
\end{equation}

\begin{theorem}\label{thm:gamma-factor-of-involution-of-representations}
	Let $\pi$ and $\sigma$ be representations of Whittaker type of $\GL_n\left(\finiteField\right)$ and $\GL_m\left(\finiteField\right)$, respectively. Then
	$$ \gammaFactor{\pi}{\sigma} = \gammaFactor[{\fieldCharacter^{-1}}]{\contragredient{\sigma}}{\contragredient{\pi}}.$$
\end{theorem}

\begin{proof}
	By definition,
	\begin{equation}\label{eq:intertwining-operator-for-involution-sigma-pi}
		U_{\outerInvolution{\pi}, \outerInvolution{\sigma}} \gspairWhittakerVector{\outerInvolution{\pi}}{\outerInvolution{\sigma}}{\fieldCharacter^{-1}} = \gammaFactor[\fieldCharacter^{-1}]{\outerInvolution{\sigma}}{\outerInvolution{\pi}} \cdot \gspairWhittakerVector{\outerInvolution{\sigma}}{\outerInvolution{\pi}}{\fieldCharacter^{-1}}.		
	\end{equation}
	Substituting \eqref{eq:pi-and-sigma-involution-whittaker-vector} and \eqref{eq:sigma-and-pi-involution-whittaker-vector} in \eqref{eq:intertwining-operator-for-involution-sigma-pi}, we get
	$$ U_{\outerInvolution{\pi}, \outerInvolution{\sigma}} S_{\sigma,\pi} \rho\left( \weyllong_{m+n} \right) \spairWhittakerVector{\sigma}{\pi} = \gammaFactor[\fieldCharacter^{-1}]{\outerInvolution{\sigma}}{\outerInvolution{\pi}} \cdot S_{\pi,\sigma} \rho\left( \weyllong_{m+n} \right) \spairWhittakerVector{\pi}{\sigma}.$$
	A simple computation shows that
	$$ U_{\outerInvolution{\pi}, \outerInvolution{\sigma}} \circ S_{\sigma,\pi} = S_{\pi,\sigma} \circ U_{\sigma, \pi}.$$
	Hence, we get that
	$$ S_{\pi,\sigma} \rho\left(\weyllong_{m+n} \right) U_{\sigma,\pi}  \spairWhittakerVector{\sigma}{\pi} = \gammaFactor[\fieldCharacter^{-1}]{\outerInvolution{\sigma}}{\outerInvolution{\pi}} \cdot S_{\pi,\sigma} \rho\left( \weyllong_{m+n} \right) \spairWhittakerVector{\pi}{\sigma},$$
	which implies that
	$$ U_{\sigma,\pi}  \spairWhittakerVector{\sigma}{\pi} = \gammaFactor[\fieldCharacter^{-1}]{\outerInvolution{\sigma}}{\outerInvolution{\pi}} \cdot \spairWhittakerVector{\pi}{\sigma}.$$
	Therefore, we must have $$\gammaFactor[\fieldCharacter^{-1}]{\outerInvolution{\sigma}}{\outerInvolution{\pi}} = \gammaFactor{\pi}{\sigma},$$
	and the statement in the theorem follows, since $\outerInvolution{\sigma} \cong \contragredient{\sigma}$ and $\outerInvolution{\pi} \cong \contragredient{\pi}$.
\end{proof}

Combining \Cref{thm:gamma-factor-of-involution-of-representations} with \Cref{thm:change-of-additive-character}, we get the following corollary.
\begin{corollary} \label{cor:relation-with-contragredient}
	Let $\pi$ and $\sigma$ be representations of Whittaker type of $\GL_n\left(\finiteField\right)$ and $\GL_m\left(\finiteField\right)$, respectively. Assume that both $\pi$ and $\sigma$ have central characters, and denote them by $\centralCharacter{\pi}$ and $\centralCharacter{\sigma}$, respectively. Then
	$$ \gammaFactor{\pi}{\sigma} = \gammaFactor{\contragredient{\sigma}}{\contragredient{\pi}} \cdot \centralCharacter{\pi}\left(-1\right)^m \centralCharacter{\sigma}\left(-1\right)^n.$$
\end{corollary}

\subsection{Multiplicativity of Gamma factors}\label{subsection:multiplicativity-of-gamma-factors}

In this section, we show that $\gammaFactor{\pi}{\sigma}$ is multiplicative.

Let $\pi$ be a representation of Whittaker type of $\GL_n\left(\finiteField\right)$. Let $m = m_1 + m_2$, and let $\sigma_1$ and $\sigma_2$ be representations of Whittaker type of $\GL_{m_1}\left(\finiteField\right)$ and $\GL_{m_2}\left(\finiteField\right)$, respectively. By \Cref{thm:unique-whittaker-vector-parabolic-induction}, the parabolic induction $\sigma_1 \parabolicInduction \sigma_2$ is also a representation of Whittaker type. Hence, the gamma factor $\gammaFactor{\pi}{(\sigma_1 \parabolicInduction \sigma_2)}$ is well defined. We will show the following theorem.

\begin{theorem}\label{thm:multiplicitavity-of-gamma-factors}
	$$ \gammaFactor{\pi}{(\sigma_1 \parabolicInduction \sigma_2)} = \gammaFactor{\pi_1}{\sigma_1} \cdot \gammaFactor{\pi_2}{\sigma_2}. $$
\end{theorem}

The proof of this theorem will occupy the remaining subsections of this section.

\begin{remark}
	Let $\sigma \subset \sigma_1 \parabolicInduction \sigma_2$ be the unique irreducible generic subrepresentation of $\sigma_1 \parabolicInduction \sigma_2$. We have that $\fieldCharacter$-Whittaker vectors of $\sigma$ are the same as $\fieldCharacter$-Whittaker vectors of $\sigma_1 \parabolicInduction \sigma_2$. Hence, \Cref{thm:multiplicitavity-of-gamma-factors} implies that $$\gammaFactor{\pi}{\sigma} = \gammaFactor{\pi}{\sigma_1} \cdot \gammaFactor{\pi}{\sigma_2}.$$
\end{remark}

Before proving the theorem, we mention two other multiplicative properties that follow immediately from the theorem.

The first property is that the gamma factor is also multiplicative in the first variable. This follows from \Cref{thm:multiplicitavity-of-gamma-factors} combined with \Cref{thm:gamma-factor-of-involution-of-representations}.
\begin{corollary}Let $\pi_1$ and $\pi_2$ be representations of Whittaker type of $\GL_{n_1}\left(\finiteField\right)$ and $\GL_{n_2}\left(\finiteField\right)$, respectively, and let $\sigma$ be a representation of Whittaker type of $\GL_m\left(\finiteField\right)$. Then
	$$\gammaFactor{(\pi_1 \parabolicInduction \pi_2)}{\sigma} = \gammaFactor{\pi_1}{\sigma} \cdot \gammaFactor{\pi_2}{\sigma}.$$
\end{corollary}

The second corollary allows us to express the gamma factor of two parabolically induced representations as the product of the gamma factors of the components of the parabolic induction. It follows by repeatedly using multiplicativity in both variables.

\begin{corollary}\label{cor:full-parabolic-induction-multiplicativity-property}
	Let $\pi_1, \dots, \pi_r$ and $\sigma_1, \dots, \sigma_t$ be irreducible generic representations of $\GL_{n_1}\left(\finiteField\right), \dots, \GL_{n_r}\left(\finiteField\right)$ and $\GL_{m_1}\left(\finiteField\right), \dots, \GL_{m_t}\left(\finiteField\right)$, respectively. Then \begin{enumerate}
		\item  We have $$ \gammaFactor{(\pi_1 \parabolicInduction \dots \parabolicInduction \pi_r)}{(\sigma_1 \parabolicInduction \dots \parabolicInduction \sigma_t)} = \prod_{i=1}^r \prod_{j=1}^t \gammaFactor{\pi_i}{\sigma_j}.$$
		\item If $\pi$ is the unique irreducible generic subrepresentation of $\pi_1 \parabolicInduction \dots \parabolicInduction \pi_r$ and $\sigma$ is the unique irreducible generic subrepresentation of $\sigma_1 \parabolicInduction \dots \parabolicInduction \sigma_t$, then
		$$ \gammaFactor{\pi}{\sigma} = \prod_{i=1}^r \prod_{j=1}^t \gammaFactor{\pi_i}{\sigma_j}. $$
	\end{enumerate}
\end{corollary}
In the next subsections we make preparations for the proof of \Cref{thm:multiplicitavity-of-gamma-factors}.

\subsubsection{Transitivity of parabolic induction}

Let $\tau_1, \tau_2$ and $\tau_3$ be finite dimensional representations of $\GL_{n_1}\left(\finiteField\right)$, $\GL_{n_2}\left(\finiteField\right)$ and $\GL_{n_3}\left(\finiteField\right)$, respectively.

We realize elements in $\left(\tau_1 \parabolicInduction \tau_2\right) \otimes \tau_3$ as functions $\GL_{n_1 + n_2}\left(\finiteField\right) \rightarrow \tau_1 \otimes \tau_2 \otimes \tau_3$ in the obvious way. Similarly, we realize elements in $\tau_1 \otimes \left(\tau_2 \parabolicInduction \tau_3\right)$ as functions $\GL_{n_2 + n_3} \rightarrow \tau_1 \otimes \tau_2 \otimes \tau_3$ in the obvious way.

Consider the space $\left(\tau_1 \parabolicInduction \tau_2\right) \parabolicInduction \tau_3$. We will regard elements of this space as functions $$f \colon \GL_{n_1+n_2+n_3}\left(\finiteField\right) \times \GL_{n_1 + n_2}\left(\finiteField\right) \rightarrow \tau_1 \otimes \tau_2 \otimes \tau_3,$$ where $f\left(g;h\right)$ means evaluating $f$ at $g \in \GL_{n_1 + n_2 + n_3}\left(\finiteField\right)$ and then evaluating the resulting function at $h \in \GL_{n_1 + n_2}\left(\finiteField\right)$. We will similarly regard elements of  $\tau_1 \parabolicInduction \left(\tau_2 \parabolicInduction \tau_3\right)$ as functions  $$f \colon \GL_{n_1+n_2+n_3}\left(\finiteField\right) \times \GL_{n_2 + n_3}\left(\finiteField\right) \rightarrow \tau_1 \otimes \tau_2 \otimes \tau_3.$$

We have an isomorphism of representations $$L_{\tau_1 , \tau_2 ; \tau_3} \colon \left(\tau_1 \parabolicInduction \tau_2\right) \parabolicInduction \tau_3 \rightarrow \tau_1 \parabolicInduction \tau_2 \parabolicInduction \tau_3,$$ given by mapping a function $f \in (\tau_1 \parabolicInduction \tau_2) \parabolicInduction \tau_3$ to 
$$L_{\tau_1 , \tau_2 ; \tau_3} f\left(g\right) = f\left(g; \identityMatrix{n_1 + n_2}\right),$$
where $g \in \GL_{n_1+n_2+n_3}\left(\finiteField\right)$.

Similarly, we have an isomorphism of representations $$L_{\tau_1; \tau_2, \tau_3} \colon \tau_1 \parabolicInduction \left(\tau_2 \parabolicInduction \tau_3\right) \rightarrow \tau_1 \parabolicInduction \tau_2 \parabolicInduction \tau_3,$$ given by mapping a function $f \in \tau_1 \parabolicInduction  (\tau_2 \parabolicInduction \tau_3) $ to $$L_{\tau_1; \tau_2, \tau_3} f \left(g\right) = f\left(g; \identityMatrix{n_2 + n_3}\right),$$
where again $g \in \GL_{n_1+n_2+n_3}\left(\finiteField\right)$.

Assume now that $\tau_1$, $\tau_2$ and $\tau_3$ have non-zero $\fieldCharacter$-Whittaker vectors, $\whittakerVector{\tau_1}$, $\whittakerVector{\tau_2}$ and $\whittakerVector{\tau_3}$, respectively, and assume that up to scalar multiplication, these Whittaker vectors are unique. We denote, as before, the following non-zero $\fieldCharacter$-Whittaker vectors $\spairWhittakerVector{\tau_1}{\tau_2} = \pairWhittakerVector{\tau_1}{\tau_2} \in \tau_1 \parabolicInduction \tau_2$ and $\spairWhittakerVector{\tau_2}{\tau_3} = \pairWhittakerVector{\tau_2}{\tau_3} \in \tau_2 \parabolicInduction \tau_3$. We also define the following non-zero $\fieldCharacter$-Whittaker vectors   $\spairWhittakerVector{\tau_1}{\tau_2 \parabolicInduction \tau_3} = f_{{\whittakerVector{\tau_1}},{\spairWhittakerVector{\tau_2}{\tau_3}}} \in \tau_1 \parabolicInduction \left(\tau_2 \parabolicInduction \tau_3\right)$ and $\spairWhittakerVector{\tau_1 \parabolicInduction \tau_2}{\tau_3} = f_{\spairWhittakerVector{\tau_1}{\tau_2}, \whittakerVector{\tau_3}} \in \left(\tau_1 \parabolicInduction \tau_2\right) \parabolicInduction \tau_3$. Finally, we define $$\tripleWhittakerVector{\tau_1}{\tau_2}{\tau_3} = L_{\tau_1;\tau_2,\tau_3} \spairWhittakerVector{\tau_1}{\tau_2 \parabolicInduction \tau_3} = L_{\tau_1, \tau_2; \tau_3} \spairWhittakerVector{\tau_1 \parabolicInduction \tau_2}{\tau_3} \in \tau_1 \parabolicInduction \tau_2 \parabolicInduction \tau_3.$$
Then $\tripleWhittakerVector{\tau_1}{\tau_2}{\tau_3}$ is the $\fieldCharacter$-Whittaker vector in $\tau_1 \parabolicInduction \tau_2 \parabolicInduction \tau_3$ supported on the double coset $\parabolicGroup_{n_1,n_2,n_3} \weylnm_{n_3,n_2,n_1} \unipotentSubgroup_{n_1 + n_2 + n_3},$ with $\tripleWhittakerVector{\tau_1}{\tau_2}{\tau_3}\left( \weylnm_{n_3,n_2,n_1} \right) = \whittakerVector{\tau_1} \otimes \whittakerVector{\tau_2} \otimes \whittakerVector{\tau_3}$, where $$\weylnm_{n_3,n_2,n_1} = \begin{pmatrix}
& & \identityMatrix{n_1}\\
& \identityMatrix{n_2}\\
\identityMatrix{n_3}
\end{pmatrix}.$$

\subsubsection{Intertwining operators}

We return to the notations of the beginning of this section. Let $\pi$ be a representation of Whittaker type of $\GL_n\left(\finiteField\right)$. Let $m = m_1 + m_2$, and let $\sigma_1$ and $\sigma_2$ be representations of Whittaker type of $\GL_{m_1}\left(\finiteField\right)$ and $\GL_{m_2}\left(\finiteField\right)$, respectively.

Using the isomorphisms from the previous section, we obtain maps such that the following diagrams are commutative.

\begin{equation}\label{eq:diagram-flatten-sigma1}
	\xymatrix{\sigma_1 \parabolicInduction \left(\sigma_2 \parabolicInduction \pi\right) \ar[rr]^{\identityMap_{\sigma_1} \otimes U_{\sigma_2, \pi}} \ar[d]^{L_{\sigma_1;\sigma_2,\pi}} & & \sigma_1 \parabolicInduction \left(\pi \parabolicInduction \sigma_2\right) \ar[d]^{L_{\sigma_1;\pi,\sigma_2}} \\
		\sigma_1 \parabolicInduction \sigma_2 \parabolicInduction \pi \ar[rr]^{\tilde{U}_{\sigma_2, \pi}} && \sigma_1 \parabolicInduction \pi \parabolicInduction \sigma_2
	},
\end{equation}

\begin{equation}\label{eq:diagram-flatten-sigma2}
	\xymatrix{ \left(\sigma_1 \parabolicInduction \pi\right) \parabolicInduction \sigma_2  \ar[rr]^{U_{\sigma_1, \pi } \otimes \identityMap_{\sigma_2}} \ar[d]^{L_{\sigma_1,\pi;\sigma_2}} && \left(\pi \parabolicInduction \sigma_1\right) \parabolicInduction \sigma_2 \ar[d]^{L_{\pi,\sigma_1;\sigma_2}} \\
		\sigma_1 \parabolicInduction \pi \parabolicInduction \sigma_2  \ar[rr]^{\tilde{U}_{\sigma_1, \pi }} && \pi \parabolicInduction \sigma_1 \parabolicInduction \sigma_2 },
\end{equation}
\begin{equation}\label{eq:diagram-flatten-pi}
	\xymatrix{\left(\sigma_1 \parabolicInduction \sigma_2\right) \parabolicInduction \pi  \ar[d]^{L_{\sigma_1,\sigma_2;\pi}} \ar[rr]^{U_{\sigma_1 \parabolicInduction \sigma_2, \pi}} & & \pi \parabolicInduction \left(\sigma_1 \parabolicInduction \sigma_2\right) \ar[d]^{L_{\pi;\sigma_1,\sigma_2}}  \\
		\sigma_1 \parabolicInduction \sigma_2 \parabolicInduction \pi \ar[rr]^{\tilde{U}_{\sigma_1 \parabolicInduction \sigma_2, \pi}} & & \pi \parabolicInduction \sigma_1 \parabolicInduction \sigma_2}.
\end{equation}

Let us explain these diagrams. We begin with explaining \eqref{eq:diagram-flatten-sigma1}. The map $\identityMap_{\sigma_1} \otimes U_{\sigma_2, \pi} \colon \sigma_1 \otimes \left(\sigma_2 \parabolicInduction \pi\right) \rightarrow \sigma_1 \otimes \left(\pi \parabolicInduction \sigma_2\right)$ is a homomorphism of representations. It defines a homomorphism $\sigma_1 \parabolicInduction \left(\sigma_2 \parabolicInduction \pi\right) \rightarrow \sigma_1 \parabolicInduction \left(\pi \parabolicInduction \sigma_2\right)$, which we keep denoting by $\identityMap_{\sigma_1} \otimes U_{\sigma_2, \pi}$. By unwrapping the definitions, we see that the map $\tilde{U}_{\sigma_2, \pi} \colon \sigma_1 \parabolicInduction \sigma_2 \parabolicInduction \pi \rightarrow \sigma_1 \parabolicInduction \pi \parabolicInduction \sigma_2$ is given by the formula
\begin{equation}\label{eq:sigma2-intertwining-operator}
	\tilde{U}_{\sigma_2,\pi}\left(f\right)\left(g\right) = \sum_{u_{n,m_2} \in \unipotentRadical_{n,m_2}} \mathrm{sw}_{\sigma_2, \pi}f\left(\begin{pmatrix}
		I_{m_1} &\\
		& & I_{m_2}\\
		& I_n
	\end{pmatrix} \begin{pmatrix}
		I_{m_1} &\\
		& u_{n,m_2}
	\end{pmatrix} g \right).
\end{equation}
The commutative diagram \eqref{eq:diagram-flatten-sigma2} is similar. We get by unwrapping the definitions that the map $\tilde{U}_{\sigma_1, \pi } \colon \sigma_1 \parabolicInduction \pi \parabolicInduction \sigma_2 \rightarrow \pi \parabolicInduction \sigma_1 \parabolicInduction \sigma_2$ is given by
\begin{equation}\label{eq:sigma1-intertwining-operator}
	\tilde{U}_{\sigma_1,\pi}\left(f\right)\left(g\right) = \sum_{u_{n,m_1} \in \unipotentRadical_{n,m_1}} \mathrm{sw}_{\sigma_1, \pi}f\left(\begin{pmatrix}
		& I_{m_1}\\
		I_n &  \\
		& & I_{m_2}
	\end{pmatrix} \begin{pmatrix}
		u_{n,m_1} &\\
		& I_{m_2}
	\end{pmatrix} g \right).
\end{equation}
Finally, in the diagram \eqref{eq:diagram-flatten-pi}, we have that $\tilde{U}_{\sigma_1 \parabolicInduction \sigma_2, \pi} \colon \sigma_1 \parabolicInduction \sigma_2 \parabolicInduction \pi \rightarrow \pi \parabolicInduction \sigma_1 \parabolicInduction \sigma_2$ is given by \begin{equation}\label{eq:sigma1-circ-sigma2-intertwining-operator}
	\tilde{U}_{\sigma_1 \parabolicInduction \sigma_2, \pi}\left(f\right)\left(g\right) = \sum_{u_{n,m} \in \unipotentRadical_{n,m}} \tilde{f}\left( \begin{pmatrix}
		& I_{m_1}\\
		& & I_{m_2}\\
		I_n
	\end{pmatrix} u_{n,m} g \right),
\end{equation}
where for $g \in \GL_{n + m}\left(\finiteField\right)$, we mean $\tilde{f}\left(g\right) = \mathrm{sw}_{\sigma_1,\pi} \mathrm{sw}_{\sigma_2, \pi} f\left(g\right)$.

\begin{proposition}
	We have $$ \tilde{U}_{\sigma_1 \parabolicInduction \sigma_2, \pi} = \tilde{U}_{\sigma_1, \pi } \circ \tilde{U}_{\sigma_2,\pi}.$$
\end{proposition}
\begin{proof}
	Let $f \in \sigma_1 \parabolicInduction \sigma_2 \parabolicInduction \pi$ and $g \in \GL_{n + m}\left(\finiteField\right)$. Then by \eqref{eq:sigma2-intertwining-operator} and \eqref{eq:sigma1-intertwining-operator},
	\begin{align*}
	(\tilde{U}_{\sigma_1, \pi } \circ \tilde{U}_{\sigma_2,\pi}f)\left(g\right) =& \sum_{X \in \Mat{n}{m_1}\left(\finiteField\right)} \sum_{Y \in \Mat{n}{m_2}\left(\finiteField\right)}\tilde{f}\left(\begin{pmatrix}
	I_{m_1} &\\
	& & I_{m_2}\\
	& I_n
	\end{pmatrix} \begin{pmatrix}
	I_{m_1} &\\
	& I_n & Y\\
	& & I_{m_2} 
	\end{pmatrix} \right.\\
	& \times \left.\begin{pmatrix}
	& I_{m_1}\\
	I_n &  \\
	& & I_{m_2}
	\end{pmatrix} \begin{pmatrix}
	I_n & X &\\
	& I_{m_1}\\
	& & I_{m_2}
	\end{pmatrix} g \right).
	\end{align*}
	A simple computation shows that
	\begin{align*}
	& \begin{pmatrix}
	I_{m_1} &\\
	& & I_{m_2}\\
	& I_n
	\end{pmatrix} \begin{pmatrix}
	I_{m_1} &\\
	& I_n & Y\\
	& & I_{m_2} 
	\end{pmatrix} 
	\begin{pmatrix}
	& I_{m_1}\\
	I_n &  \\
	& & I_{m_2}
	\end{pmatrix} \begin{pmatrix}
	I_n & X &\\
	& I_{m_1}\\
	& & I_{m_2}
	\end{pmatrix}\\	
	=&  \begin{pmatrix}
	& I_{m_1} &\\
	& & I_{m_2}\\
	I_n & &
	\end{pmatrix} \begin{pmatrix}
	I_n & X & Y\\
	& I_{m_1}\\
	& & I_{m_2}
	\end{pmatrix}.
	\end{align*}
	Hence, we get $$(\tilde{U}_{\sigma_1, \pi } \circ \tilde{U}_{\sigma_2,\pi}f)\left(g\right) = \sum_{X \in \Mat{n}{m_1}\left(\finiteField\right)} \sum_{Y \in \Mat{n}{m_2}\left(\finiteField\right)}\tilde{f}\left(\begin{pmatrix}
	& I_{m_1} &\\
	& & I_{m_2}\\
	I_n & &
	\end{pmatrix} \begin{pmatrix}
	I_n & X & Y\\
	& I_{m_1}\\
	& & I_{m_2}
	\end{pmatrix} g\right),$$
	and the last sum is $\tilde{U}_{\sigma_1 \parabolicInduction \sigma_2, \pi}\left(f\right)\left(g\right)$ by \eqref{eq:sigma1-circ-sigma2-intertwining-operator}.
\end{proof}

\subsubsection{Proof of \Cref{thm:multiplicitavity-of-gamma-factors}}
Let $\whittakerVector{\pi}$, $\whittakerVector{\sigma_1}$ and $\whittakerVector{\sigma_2}$ be non-zero $\fieldCharacter$-Whittaker vectors of $\pi$, $\sigma_1$ and $\sigma_2$, respectively. We keep the notations from the previous section. We are now ready to prove \Cref{thm:multiplicitavity-of-gamma-factors}.

\begin{proof}
	By definition, we have $$ \left(\identityMap_{\sigma_1} \otimes U_{\sigma_2, \pi}\right)\left( \spairWhittakerVector{\sigma_1}{\sigma_2 \parabolicInduction \pi} \right) = \gammaFactor{\pi}{\sigma_2} \spairWhittakerVector{\sigma_1}{\pi \parabolicInduction \sigma_2}.$$
Since $L_{\sigma_1;\sigma_2,\pi} \spairWhittakerVector{\sigma_1}{\sigma_2 \parabolicInduction \pi} = \tripleWhittakerVector{\sigma_1}{\sigma_2}{\pi}$ and $L_{\sigma_1;\pi,\sigma_2} \spairWhittakerVector{\sigma_1}{\pi \parabolicInduction \sigma_2} = \tripleWhittakerVector{\sigma_1}{\pi}{\sigma_2}$, we get from the commutative diagram \eqref{eq:diagram-flatten-sigma1} that
$$ \tilde{U}_{\sigma_2, \pi}  \tripleWhittakerVector{\sigma_1}{\sigma_2}{\pi} = \gammaFactor{\pi}{\sigma_2} \tripleWhittakerVector{\sigma_1}{\pi}{\sigma_2}.$$
Similarly, we have that $$ \left(U_{\sigma_1, \pi} \otimes  \identityMap_{\sigma_2}\right)\left( \spairWhittakerVector{\sigma_1 \parabolicInduction \pi}{\sigma_2} \right) = \gammaFactor{\pi}{\sigma_1} \spairWhittakerVector{\pi \parabolicInduction \sigma_1}{\sigma_2},$$ and we get from the commutative diagram \eqref{eq:diagram-flatten-sigma2} that 
$$ \tilde{U}_{\sigma_1, \pi}  \tripleWhittakerVector{\sigma_1}{\pi}{\sigma_2} = \gammaFactor{\pi}{\sigma_1} \tripleWhittakerVector{\pi}{\sigma_1}{\sigma_2}.$$
Finally, we have $$U_{\sigma_1 \parabolicInduction \sigma_2, \pi} \spairWhittakerVector{\sigma_1 \parabolicInduction \sigma_2}{\pi} = \gammaFactor{\pi}{(\sigma_1 \parabolicInduction \sigma_2)} \spairWhittakerVector{\pi}{\sigma_1 \parabolicInduction \sigma_2}.$$
Since $L_{\sigma_1,\sigma_2;\pi} \spairWhittakerVector{\sigma_1 \parabolicInduction \sigma_2}{\pi} = \tripleWhittakerVector{\sigma_1}{\sigma_2}{\pi}$ and $L_{\pi; \sigma_1, \sigma_2} \spairWhittakerVector{\pi}{\sigma_1 \parabolicInduction \sigma_2} = \tripleWhittakerVector{\pi}{\sigma_1}{\sigma_2}$, we get from the commutative diagram \eqref{eq:diagram-flatten-pi}
$$\tilde{U}_{\sigma_1 \parabolicInduction \sigma_2, \pi} \tripleWhittakerVector{\sigma_1}{\sigma_2}{\pi} = \gammaFactor{\pi}{(\sigma_1 \parabolicInduction \sigma_2)} \tripleWhittakerVector{\pi}{\sigma_1}{\sigma_2}.$$
Since $\tilde{U}_{\sigma_1 \parabolicInduction \sigma_2, \pi} = \tilde{U}_{\sigma_1, \pi} \circ \tilde{U}_{\sigma_2, \pi}$, we get that $$\gammaFactor{\pi}{(\sigma_1 \parabolicInduction \sigma_2)} \tripleWhittakerVector{\pi}{\sigma_1}{\sigma_2} = \gammaFactor{\pi}{\sigma_1} \gammaFactor{\pi}{\sigma_2} \tripleWhittakerVector{\pi}{\sigma_1}{\sigma_2},$$
and the theorem follows.
\end{proof}

\subsection{Expression in terms of Bessel functions}

In this section, we express the Shahidi gamma factor of two irreducible generic representations in terms of their Bessel functions.

Let $\pi$ and $\sigma$ be irreducible generic representations of $\GL_n\left(\finiteField\right)$ and $\GL_m\left(\finiteField\right)$, respectively. We assume that $\pi$ and $\sigma$ are realized by their Whittaker models $\whittaker\left(\pi, \fieldCharacter\right)$ and $\whittaker\left(\sigma, \fieldCharacter\right)$, respectively. We choose the Whittaker vectors of $\pi$ and $\sigma$ to be their corresponding Bessel functions, i.e., we choose $\whittakerVector{\pi} = \repBesselFunction{\pi}$ and $\whittakerVector{\sigma} = \repBesselFunction{\sigma}$. We denote $\pairBesselFunction{\sigma}{\pi} = \spairWhittakerVector{\sigma}{\pi} = f_{\repBesselFunction{\sigma},\repBesselFunction{\pi}}$ and similarly $\pairBesselFunction{\pi}{\sigma} = \spairWhittakerVector{\pi}{\sigma} = f_{\repBesselFunction{\pi},\repBesselFunction{\sigma}}$.

Assume that $n \ge m$. By definition, we have that for any $g \in \GL_{n + m}\left(\finiteField\right)$, $$\left(U_{\sigma, \pi} \pairBesselFunction{\sigma}{\pi}\right)\left(g\right) = \gammaFactor{\pi}{\sigma} \pairBesselFunction{\pi}{\sigma}\left(g\right).$$ 
Substituting $g = \weylnm_{m,n}$, we get
$$\gammaFactor{\pi}{\sigma} \pairBesselFunction{\pi}{\sigma} \left(\weylnm_{m,n}\right) = \sum_{u \in \unipotentRadical_{n,m}} \swapHomomorphism_{\sigma,\pi}{\pairBesselFunction{\sigma}{\pi}}\left( \weylnm_{n,m} u \weylnm_{m,n} \right),$$ and therefore
\begin{equation}\label{eq:gamma-factor-as-sum-of-bessel-functions-first-formula}
	\gammaFactor{\pi}{\sigma} \repBesselFunction{\pi} \otimes \repBesselFunction{\sigma} = \sum_{A \in \Mat{n}{m}\left(\finiteField\right)} \swapHomomorphism_{\sigma,\pi}{\pairBesselFunction{\sigma}{\pi}} \begin{pmatrix}
		\identityMatrix{m} & \\
		A & \identityMatrix{n}
	\end{pmatrix}.
\end{equation}

In order for $\pairBesselFunction{\sigma}{\pi} \begin{pmatrix}
\identityMatrix{m} & \\
A & \identityMatrix{n}
\end{pmatrix}$ not to vanish, we must have $\begin{pmatrix}
\identityMatrix{m} & \\
A & \identityMatrix{n}
\end{pmatrix} \in \parabolicGroup_{m,n} \weylnm_{n,m} \unipotentSubgroup_{n + m}$, so there must exist $\begin{pmatrix}
p_1 & x\\
& p_2
\end{pmatrix} \in \parabolicGroup_{m,n}$ and $\begin{pmatrix}
u_1 & y\\
& u_2
\end{pmatrix} \in \unipotentSubgroup_{n + m}$, where $u_1 \in \unipotentSubgroup_n$ and $u_2 \in \unipotentSubgroup_m$, such that $$\begin{pmatrix}
p_1 & x\\
& p_2
\end{pmatrix} \begin{pmatrix}
\identityMatrix{m} & \\
A & \identityMatrix{n}
\end{pmatrix} = \weylnmmatrix{m}{n} \begin{pmatrix}
u_1 & y\\
& u_2
\end{pmatrix},$$
i.e.,
$$\begin{pmatrix}
p_1 + xA & x\\
p_2 A & p_2
\end{pmatrix} = \begin{pmatrix}
 & u_2 \\
u_1 & y
\end{pmatrix}.$$
Therefore, we have $p_1 + xA = 0$ and $x = \left(0_{m \times \left(n-m\right)}, u_2\right)$.

In order to proceed, we will separate two cases, the case where $n > m$ and the case where $n = m$.

\subsubsection{The case $n > m$}\label{subsection:bessel-expression-m-smaller-than-n}

In this case, we write $A = \begin{pmatrix}
A_1\\
A_2
\end{pmatrix}$, where $A_1 \in \Mat{\left(n-m\right)}{m}\left(\finiteField\right)$ and $A_2 \in \Mat{m}{m}\left(\finiteField\right)$ and $x = \left(0_{m \times \left(n-m\right)}, u_2 \right)$. Then $p_1 + xA = 0$ implies $p_1 + u_2 A_2 = 0$, and therefore $A_2$ is invertible.

Write \begin{align*}
	\begin{pmatrix}
	\identityMatrix{n} &\\
	A & \identityMatrix{m}
	\end{pmatrix} = \begin{pmatrix}
	I_m &\\
	A_1 & I_{n-m} &\\
	A_2 & & I_m
	\end{pmatrix} &= \begin{pmatrix}
	I_m & & -A_2^{-1}\\
	A_1 & I_{n-m} & \\
	A_2 & & 
	\end{pmatrix} \begin{pmatrix}
	I_m & & A_2^{-1} \\
	& I_{n-m} & -A_1 A_2^{-1}\\
	& & I_m
	\end{pmatrix}\\
%	&= \begin{pmatrix}
%	I_m & & -A_2^{-1}\\
%	A_1 & I_{n-m} & \\
%	A_2 & & 
%	\end{pmatrix} \begin{pmatrix}
%	& \identityMatrix{m} &  \\
%	& & I_{n-m} \\
%	I_m & &
%	\end{pmatrix} \weylnm_{n,m} \begin{pmatrix}
%	I_m & & A_2^{-1} \\
%	& I_{n-m} & -A_1 A_2^{-1}\\
%	& & I_m
%	\end{pmatrix} \\
	&= \begin{pmatrix}
	-A_2^{-1} & I_m & \\
	 & A_1 &  I_{n-m}\\
	 & A_2 & 
	\end{pmatrix} \weylnm_{n,m} \begin{pmatrix}
	I_m & & A_2^{-1} \\
	& I_{n-m} & -A_1 A_2^{-1}\\
	& & I_m
	\end{pmatrix}.
\end{align*}
Therefore, we have 
\begin{equation}\label{eq:gamma-factor-term-m-smaller-than-n}
	{\pairBesselFunction{\sigma}{\pi}} \begin{pmatrix}
		\identityMatrix{m} & \\
		A & \identityMatrix{n}
	\end{pmatrix} = \fieldCharacter\begin{pmatrix}
		I_{n-m} & -A_1 A_2^{-1}\\
		& I_m
	\end{pmatrix} \sigma\left(-A_2^{-1}\right) \otimes \pi \begin{pmatrix}
		A_1 &  I_{n-m}\\
		A_2
	\end{pmatrix}  \repBesselFunction{\sigma} \otimes \repBesselFunction{\pi}.
\end{equation}

Substituting \eqref{eq:gamma-factor-term-m-smaller-than-n} back in \eqref{eq:gamma-factor-as-sum-of-bessel-functions-first-formula}, we get

\begin{equation}\label{eq:gamma-intermediate-m-smaller-than-n}
	\begin{split}
			&\gammaFactor{\pi}{\sigma} \repBesselFunction{\pi} \otimes \repBesselFunction{\sigma} 
		\\ =& \sum_{\substack{A_1 \in \Mat{\left(n-m\right)}{m}\left(\finiteField\right)\\
				A_2 \in \GL_m\left(\finiteField\right)}} \fieldCharacter\begin{pmatrix}
			I_{n-m} & -A_1 A_2^{-1}\\
			& I_m
		\end{pmatrix} \pi \begin{pmatrix}
			A_1 &  I_{n-m}\\
			A_2
		\end{pmatrix} \otimes \sigma\left(-A_2^{-1}\right) \repBesselFunction{\pi} \otimes \repBesselFunction{\sigma}.
	\end{split}
\end{equation}

We evaluate both sides of \eqref{eq:gamma-intermediate-m-smaller-than-n} at $\left(\identityMatrix{n}, \identityMatrix{m}\right)$ to get

$$\gammaFactor{\pi}{\sigma} = \sum_{\substack{A_1 \in \Mat{\left(n-m\right)}{m}\left(\finiteField\right)\\
		A_2 \in \GL_m\left(\finiteField\right)}} \fieldCharacter\begin{pmatrix}
I_{n-m} & -A_1 A_2^{-1}\\
& I_m
\end{pmatrix} \repBesselFunction{\pi} \begin{pmatrix}
A_1 &  I_{n-m}\\
A_2
\end{pmatrix} \repBesselFunction{\sigma} \left(-A_2^{-1}\right).$$

Writing $$\begin{pmatrix}
A_1 &  I_{n-m}\\
A_2
\end{pmatrix} = \begin{pmatrix}
I_{n-m} & A_1 A_2^{-1}\\
& I_{m}
\end{pmatrix} \begin{pmatrix}
& I_{n-m}\\
A_2
\end{pmatrix},$$
we get 
$$\repBesselFunction{\pi} \begin{pmatrix}
A_1 &  I_{n-m}\\
A_2
\end{pmatrix} = \fieldCharacter \begin{pmatrix}
I_{n-m} & A_1 A_2^{-1}\\
& I_{m}
\end{pmatrix} \repBesselFunction{\pi} \begin{pmatrix}
& I_{n-m}\\
A_2
\end{pmatrix},$$
and therefore
\begin{align*}
	\gammaFactor{\pi}{\sigma} &= \sum_{\substack{A_1 \in \Mat{\left(n-m\right)}{m}\left(\finiteField\right)\\
			A_2 \in \GL_m\left(\finiteField\right)}}  \repBesselFunction{\pi} \begin{pmatrix}
	&  I_{n-m}\\
	A_2
	\end{pmatrix} \repBesselFunction{\sigma} \left(-A_2^{-1}\right).
\end{align*}
The summand is independent of $A_1$. Using the equivariance properties of the Bessel function, we get that the summand is invariant under $\unipotentSubgroup_{m}$ left translations of $A_2$. Finally, using the properties of the Bessel function discussed in \Cref{section:whittaker-models-and-bessel-functions}, we get
$$\gammaFactor{\pi}{\sigma} = q^{\frac{m \left(2n-m-1\right)}{2}} \centralCharacter{\sigma}\left(-1\right) \sum_{x \in \unipotentSubgroup_{m} \backslash \GL_m\left(\finiteField\right)} \repBesselFunction{\pi} \begin{pmatrix}
	&  I_{n-m}\\
	x
\end{pmatrix} \grepBesselFunction{\contragredient{\sigma}}{\fieldCharacter^{-1}} \left(x\right),$$
where $q^{\frac{m \left(2n-m-1\right)}{2}} = \sizeof{\Mat{\left(n-m\right)}{m}\left(\finiteField\right)} \cdot \sizeof{\unipotentSubgroup_{m}}$.

\subsubsection{The case $n=m$}\label{subsection:bessel-expression-m-equals-n}
In this case, we have $-p_1 = xA$, and therefore $A$ is invertible. We write $$\begin{pmatrix}
I_n & \\
A & I_n
\end{pmatrix} = \begin{pmatrix}
I_n & -A^{-1}\\
A & 
\end{pmatrix} \begin{pmatrix}
I_n & A^{-1}\\
& I_n
\end{pmatrix} = \begin{pmatrix}
-A^{-1} & I_n\\
& A
\end{pmatrix} \weylnm_{n,n} \begin{pmatrix}
I_n & A^{-1}\\
& I_n
\end{pmatrix}.$$
Therefore, we have \begin{equation}\label{eq:gamma-factor-term-m-equals-n}
	{\pairBesselFunction{\sigma}{\pi}} \begin{pmatrix}
		\identityMatrix{m} & \\
		A & \identityMatrix{n}
	\end{pmatrix} = \fieldCharacter \begin{pmatrix}
		I_n & A^{-1}\\
		& I_n
	\end{pmatrix} \sigma\left(-A^{-1}\right)  \otimes \pi\left(A\right) \repBesselFunction{\sigma} \otimes \repBesselFunction{\pi}.
\end{equation}
Substituting \eqref{eq:gamma-factor-term-m-equals-n} in \eqref{eq:gamma-factor-as-sum-of-bessel-functions-first-formula}, we get
\begin{equation}\label{eq:gamma-intermediate-m-equals-n}
	\gammaFactor{\pi}{\sigma} \repBesselFunction{\pi} \otimes \repBesselFunction{\sigma} = \sum_{A \in \GL_n\left(\finiteField\right)} \fieldCharacter \begin{pmatrix}
		I_n & A^{-1}\\
		& I_n
	\end{pmatrix} \pi\left(A\right) \otimes \sigma\left(-A^{-1}\right) \repBesselFunction{\pi} \otimes \repBesselFunction{\sigma}.
\end{equation}
Evaluating both sides of \eqref{eq:gamma-intermediate-m-equals-n} at $\left(\identityMatrix{n}, \identityMatrix{n}\right)$, we get
$$\gammaFactor{\pi}{\sigma} = \sum_{A \in \GL_n\left(\finiteField\right)} \fieldCharacter \begin{pmatrix}
I_n & A^{-1}\\
& I_n
\end{pmatrix} \repBesselFunction{\pi} \left(A\right) \repBesselFunction{\sigma}\left(-A^{-1}\right).$$
The summand is invariant under $\unipotentSubgroup_{n}$ left translations. Using the properties of the Bessel function discussed in \Cref{section:whittaker-models-and-bessel-functions}, we get
$$\gammaFactor{\pi}{\sigma} = q^{\frac{n\left(n-1\right)}{2}} \centralCharacter{\sigma}\left(-1\right) \sum_{x \in \unipotentSubgroup_{n} \backslash \GL_n\left(\finiteField\right)} \fieldCharacter \begin{pmatrix}
I_n & x^{-1}\\
& I_n
\end{pmatrix} \repBesselFunction{\pi} \left(x\right) \grepBesselFunction{\contragredient{\sigma}}{\fieldCharacter^{-1}}\left(x\right).$$

\subsubsection{Summary of cases}

We conclude this section by writing down formulas for the Shahidi gamma factor for a pair of irreducible generic representations, in terms of their Bessel functions for all cases. In order to do that, we use \Cref{thm:gamma-factor-of-involution-of-representations} and the formulas from Sections \ref{subsection:bessel-expression-m-smaller-than-n} and \ref{subsection:bessel-expression-m-equals-n}.

\begin{theorem}\label{thm:explicit-formula-for-gamma-factors-in-terms-of-bessel-functions}
	Let $\pi$ and $\sigma$ be irreducible generic representations of $\GL_n\left(\finiteField\right)$ and $\GL_m\left(\finiteField\right)$, respectively.
	
	\begin{enumerate}
		\item If $n > m$, then $$\gammaFactor{\pi}{\sigma} = q^{\frac{m \left(2n-m-1\right)}{2}} \centralCharacter{\sigma}\left(-1\right) \sum_{x \in \unipotentSubgroup_{m} \backslash \GL_m\left(\finiteField\right)} \repBesselFunction{\pi} \begin{pmatrix}
		&  I_{n-m}\\
		x
		\end{pmatrix} \grepBesselFunction{\contragredient{\sigma}}{\fieldCharacter^{-1}} \left(x\right).$$
		\item If $n = m$, then $$\gammaFactor{\pi}{\sigma} = q^{\frac{n\left(n-1\right)}{2}} \centralCharacter{\sigma}\left(-1\right) \sum_{x \in \unipotentSubgroup_{n} \backslash \GL_n\left(\finiteField\right)} \fieldCharacter \begin{pmatrix}
		I_n & x^{-1}\\
		& I_n
		\end{pmatrix} \repBesselFunction{\pi} \left(x\right) \grepBesselFunction{\contragredient{\sigma}}{\fieldCharacter^{-1}}\left(x\right).$$
		\item If $n < m$, then $$\gammaFactor{\pi}{\sigma} = q^{\frac{n \left(2m-n-1\right)}{2}} \centralCharacter{\pi}\left(-1\right) \sum_{x \in \unipotentSubgroup_{n} \backslash \GL_n\left(\finiteField\right)} \repBesselFunction{\pi} \left(x\right) \grepBesselFunction{\contragredient{\sigma}}{\fieldCharacter^{-1}} \begin{pmatrix}
		&  I_{m-n}\\
		x
		\end{pmatrix} .$$
	\end{enumerate}	
\end{theorem}

\Cref{thm:explicit-formula-for-gamma-factors-in-terms-of-bessel-functions} allows us to give a relation between the Jacquet--Piatetski-Shapiro--Shalika gamma factors defined in \Cref{section:rankin-selberg-gamma-factors} and the Shahidi gamma factor. By \Cref{prop:rankin-selberg-gamma-factor-in-terms-of-bessel-functions} and \Cref{prop:rankin-selberg-gamma-factor-m-equals-n-in-terms-of-bessel-functions}, we get the following corollary.

\begin{corollary}\label{cor:relation-between-intertwining-and-rankin-selberg}
	Let $\pi$ be an irreducible cuspidal representation of $\GL_n\left(\finiteField\right)$ and let $\sigma$ be an irreducible generic representation of $\GL_m\left(\finiteField\right)$. Then we have the equality
	$$\gammaFactor{\pi}{\sigma} = q^{\frac{m \left(2n - m - 1\right)}{2}} \centralCharacter{\sigma}\left(-1\right) \rsGammaFactor{\pi}{\contragredient{\sigma}}$$
	in either of the following cases:
	\begin{enumerate}
		\item $n > m$.
		\item $n = m$ and $\sigma$ is cuspidal. 
	\end{enumerate}
\end{corollary}

\section{Applications}

\subsection{Quantitative interpretation of gamma factors}\label{sec:quantitative interpretation-of-gamma-factors}

In this section, we give a representation theoretic interpretation of the absolute value of the Shahidi gamma factor. Our results relate the absolute value of a normalized version of the Shahidi gamma factor with the cuspidal support of the representations.

For irreducible generic representations $\pi$ and $\sigma$ of $\GL_n\left(\finiteField\right)$ and $\GL_m\left(\finiteField\right)$, respectively, we define the normalized Shahidi gamma factor by
$$\normalizedGammaFactor{\pi}{\sigma} = q^{-\frac{nm}{2}} \gammaFactor{\pi}{\sigma}.$$

It follows from \Cref{cor:full-parabolic-induction-multiplicativity-property} that under this normalization, the gamma factor is still multiplicative, i.e., the following proposition holds.

\begin{proposition}\label{prop:normalized-gamma-factor-full-multiplicativity}
	Let $\pi_1, \dots, \pi_r$ and $\sigma_1, \dots, \sigma_t$ be irreducible generic representations of $\GL_{n_1}\left(\finiteField\right), \dots, \GL_{n_r}\left(\finiteField\right)$ and $\GL_{m_1}\left(\finiteField\right), \dots, \GL_{m_t}\left(\finiteField\right)$. Suppose that $\pi$ is the unique irreducible generic subrepresentation of $\pi_1 \parabolicInduction \dots \parabolicInduction \pi_r$ and that $\sigma$ is the unique irreducible generic subrepresentation of $\sigma_1 \parabolicInduction \dots \parabolicInduction \sigma_t$. Then
	$$\normalizedGammaFactor{\pi}{\sigma} = \prod_{i = 1}^r \prod_{j = 1}^t \normalizedGammaFactor{\pi_i}{\sigma_j}.$$
\end{proposition}

By Corollaries \ref{cor:relation-between-intertwining-and-rankin-selberg} and \ref{prop:size-of-rs-gamma-factor-m-smaller-than-n} and \Cref{prop:size-of-rs-gamma-factor-m-equals-n}, we have the following proposition, which allows us to express the size of the absolute value of $\gammaFactor{\pi}{\sigma}$ where $\pi$ and $\sigma$ are cuspidal.

\begin{proposition}\label{prop:normalized-gamma-factor-size-for-cuspidal-representations}
	Let $\pi$ and $\sigma$ be irreducible cuspidal representations of $\GL_n\left(\finiteField\right)$ and $\GL_m\left(\finiteField\right)$, respectively. Then
	\begin{equation*}
		\abs{\normalizedGammaFactor{\pi}{\sigma}} = \begin{dcases}
			q^{-\frac{n}{2}} & n = m \text{ and } \pi \cong \sigma, \\
			1 & \text{ otherwise}.			
		\end{dcases}
	\end{equation*}
\end{proposition} 

\Cref{prop:normalized-gamma-factor-size-for-cuspidal-representations} tells us that the size of the normalized Shahidi gamma factor serves as a ``Kronecker delta function'' for cuspidal representations. It could be thought of an analog of \cite[Section 8.1]{Jacquet1983rankin}. Combining this with the multiplicativity property, we get the following theorem, that allows us to recover the cuspidal support of an irreducible generic representation $\pi$ by computing $\abs{\normalizedGammaFactor{\pi}{\sigma}}$ for any irreducible cuspidal $\sigma$.

\begin{theorem}\label{thm:cuspidal-support-is-determined-by-gamma-factor-absolute-value}
	Let $\pi$ be an irreducible generic representation of $\GL_n\left(\finiteField\right)$ and let $\sigma$ be an irreducible cuspidal representation of $\GL_m\left(\finiteField\right)$. Then 
	$$ \abs{\normalizedGammaFactor{\pi}{\sigma}} = q^{-\frac{{d_{\pi}\left(\sigma\right) m}}{2}},$$
	where $d_{\pi}\left(\sigma\right)$ is the number of times that $\sigma$ appears in the cuspidal support of $\pi$.
\end{theorem}
\begin{proof}
	Suppose that the cuspidal support of $\pi$ is $\left\{\pi_1, \dots, \pi_r \right\}.$ Then $\pi$ is the unique irreducible generic subrepresentation of $\pi_1 \parabolicInduction \dots \parabolicInduction \pi_r$. The result now follows immediately from \Cref{prop:normalized-gamma-factor-full-multiplicativity} and \Cref{prop:normalized-gamma-factor-size-for-cuspidal-representations}.
\end{proof}

As a corollary, we get the following converse theorem, which allows us to determine whether generic representations of $\GL_n\left(\finiteField\right)$ and $\GL_m\left(\finiteField\right)$ are isomorphic based on the absolute value of their normalized gamma factors. It is an analog of \cite[Lemma A.6]{atobe2017local}, but our proof is on the ``group side'' rather than on the ``Galois side''.

\begin{theorem}
	Let $\pi_1$ and $\pi_2$ be irreducible generic representations of $\GL_{n_1}\left(\finiteField\right)$ and $\GL_{n_2}\left(\finiteField\right)$, respectively. Suppose that for every $m > 0$ and every irreducible cuspidal representation $\sigma$ of $\GL_m\left(\finiteField\right)$ we have
	$$ \abs{\normalizedGammaFactor{\pi_1}{\sigma}} = \abs{\normalizedGammaFactor{\pi_2}{\sigma}}.$$ 
	Then $n_1 = n_2$ and $\pi_1 \cong \pi_2$.
\end{theorem}
\begin{proof}
	By \Cref{thm:cuspidal-support-is-determined-by-gamma-factor-absolute-value}, $\pi_1$ and $\pi_2$ have the same cuspidal support. By \Cref{thm:unique-whittaker-vector-parabolic-induction}, there exists a unique irreducible generic representation with a given cuspidal support.
\end{proof}

As another corollary, we explain that the functional equations in \Cref{thm:functional-equation-n-greater-than-m} and \Cref{prop:extended-functional-equation-n-equals-m} fail for $\pi$ and $\sigma$, whenever the cuspidal support of $\pi$ has a non-empty intersection with the cuspidal support of $\contragredient{\sigma}$.

\begin{corollary}
	Suppose that $\pi$ and $\sigma$ are irreducible generic representations of $\GL_n\left(\finiteField\right)$ and $\GL_m\left(\finiteField\right)$, respectively, and that $n > m$ (respectively, $n = m$). Suppose that the cuspidal support of $\pi$ has a non-empty intersection with the cuspidal support of $\contragredient{\sigma}$. Then the functional equation in \Cref{thm:functional-equation-n-greater-than-m} (respectively, \Cref{thm:rankin-selberg-m-equals-n-gamma-factor-definition}) does not hold for $\pi$ and $\sigma$.
\end{corollary}
\begin{proof}	
	If the functional equation holds for $\pi$ and $\sigma$, then it also holds for $\contragredient{\pi}$ and $\contragredient{\sigma}$. This can be seen by applying complex conjugation to the functional equation, which sends the $\fieldCharacter$-Whittaker functions to $\fieldCharacter^{-1}$-Whittaker functions of the contragredient. As in \Cref{prop:size-of-rs-gamma-factor-m-smaller-than-n} (respectively, \Cref{cor:size-of-rankin-selberg-m-equals-n}), we get that $\abs{\rsGammaFactor{\pi}{\sigma}} = q^{-\frac{m\left(n-m-1\right)}{2}}$. Whenever $\rsGammaFactor{\pi}{\sigma}$ is defined, it is given by the formula in \Cref{prop:rankin-selberg-gamma-factor-in-terms-of-bessel-functions} (respectively, \Cref{prop:rankin-selberg-gamma-factor-m-equals-n-in-terms-of-bessel-functions}), and therefore the formula in \Cref{cor:relation-between-intertwining-and-rankin-selberg} holds.
	This implies that $\abs{\normalizedGammaFactor{\pi}{\contragredient{\sigma}}} = 1$.
	
	On the other hand, because $\pi$ and $\contragredient{\sigma}$ have common elements in their cuspidal support, we have that $\abs{\normalizedGammaFactor{\pi}{\contragredient{\sigma}}} < 1$.
\end{proof}

\begin{remark}
	In his unpublished manuscript \cite{soudry1979}, the first author showed that whenever the cuspidal support of $\pi$ does not intersect the cuspidal support of $\contragredient{\sigma}$, the relevant functional equation holds. Due to length considerations, we do not include the proofs here.
\end{remark}

\subsection{Consequences for the converse theorem}

Our results from \Cref{sec:quantitative interpretation-of-gamma-factors} allow us to improve Nien's results regarding the converse theorem for irreducible generic representations of finite general linear groups.

Nien showed in \cite{Nien14} the following theorem.

\begin{theorem}\label{thm:nien-converse-thm}
	Let $\pi_1$ and $\pi_2$ be two irreducible cuspidal representations of $\GL_n\left(\finiteField\right)$ with the same central character. Suppose that for every $1 \le m \le \frac{n}{2}$, and every irreducible generic representation $\sigma$ of $\GL_m\left(\finiteField\right)$ we have \begin{equation}\label{eq:converse-theorem-equality}
		\rsGammaFactor{\pi_1}{\sigma} = \rsGammaFactor{\pi_2}{\sigma}.
	\end{equation}
	Then $\pi_1 \cong \pi_2$.
\end{theorem}

Using our results and \Cref{thm:nien-converse-thm}, we are able to deduce the following converse theorem, where $\pi_1$ and $\pi_2$ can be arbitrary generic representations (rather than just cuspidal representations), and \eqref{eq:converse-theorem-equality} needs to be verified only for cuspidal representations $\sigma$ (rather than for all generic representations). This is similar to \cite[Section 2.4]{jiang2015towards}.

\begin{theorem}
	Let $\pi_1$ and $\pi_2$ be two irreducible generic representations of $\GL_n\left(\finiteField\right)$ with the same central character. Suppose that for every $1 \le m \le \frac{n}{2}$, and every irreducible cuspidal representation $\sigma$ of $\GL_m\left(\finiteField\right)$ we have \begin{equation}\label{eq:normalized-converse-equality}
		\normalizedGammaFactor{\pi_1}{\sigma} = \normalizedGammaFactor{\pi_2}{\sigma}.
	\end{equation}
	Then $\pi_1 \cong \pi_2$.
\end{theorem}
\begin{proof}
	Our proof is by induction on the cardinality of the cuspidal support of $\pi_1$.
	
	We first notice that by \Cref{prop:normalized-gamma-factor-full-multiplicativity}, we have that for any $1 \le m \le \frac{n}{2}$ and any irreducible generic representation $\sigma$ of $\GL_m\left(\finiteField\right)$,
	$$\normalizedGammaFactor{\pi_1}{\sigma} = \normalizedGammaFactor{\pi_2}{\sigma}.$$
	
	Suppose that $\pi_1$ is cuspidal, then its cuspidal support is of cardinality $1$. If $\pi_2$ is not cuspidal, then its cuspidal support contains an irreducible cuspidal representation $\tau$ of $\GL_k\left(\finiteField\right)$, where $k \le \frac{n}{2}$. Since $k < n$, we have by \Cref{thm:cuspidal-support-is-determined-by-gamma-factor-absolute-value} that $\abs{\normalizedGammaFactor{\pi_1}{\sigma}}= 1$. We also have by \Cref{thm:cuspidal-support-is-determined-by-gamma-factor-absolute-value} that $\abs{\normalizedGammaFactor{\pi_2}{\sigma}}< 1$, which is a contraction. Therefore, $\pi_2$ is also cuspidal, and by \Cref{cor:relation-between-intertwining-and-rankin-selberg} and \Cref{thm:nien-converse-thm}, we have that $\pi_1$ and $\pi_2$ are isomorphic.
	
	Suppose now that $\pi_1$ is not cuspidal. Let $\{ \tau_1, \dots, \tau_r \}$ the cuspidal support of $\pi_1$ and let $\{ \tau'_1, \dots, \tau'_{r'} \}$ be the cuspidal support of $\pi_2$. Without loss of generality, we have that $\tau_1$ is an irreducible cuspidal representation of $\GL_{n_1}\left(\finiteField\right)$, where $n_1 \le \frac{n}{2}$. Then by \Cref{thm:cuspidal-support-is-determined-by-gamma-factor-absolute-value} we have that 
	$\abs{\normalizedGammaFactor{\pi_1}{\tau_1}} < 1$. Since $n_1 \le \frac{n}{2}$, we have that $\normalizedGammaFactor{\pi_1}{\tau_1} = \normalizedGammaFactor{\pi_2}{\tau_1}$, and therefore $\abs{\normalizedGammaFactor{\pi_2}{\tau_1}} < 1$. By \Cref{thm:cuspidal-support-is-determined-by-gamma-factor-absolute-value}, this implies that $\tau_1$ is in the cuspidal support of $\pi_2$. Without loss of generality, we may assume that $\tau'_1 = \tau_1$. By \Cref{prop:normalized-gamma-factor-full-multiplicativity}, we deduce that for any irreducible generic representation $\sigma$ of $\GL_m\left(\finiteField\right)$ where $m \le \frac{n}{2}$, \begin{equation}\label{eq:converse-theorem-condition-for-smaller-cuspidal-support}
		 \prod_{j = 2}^r \normalizedGammaFactor{\tau_j}{\sigma} = \prod_{j = 2}^{r'} \normalizedGammaFactor{\tau'_j}{\sigma}.
	\end{equation}
	Let $\pi'_1$ be the unique irreducible generic representation of $\GL_{n - n_1}\left(\finiteField\right)$ with cuspidal support $\{\tau_2,\dots,\tau_r\}$, and let $\pi'_2$ be the unique irreducible generic representation of $\GL_{n - n_1}\left(\finiteField\right)$ with cuspidal support $\{ \tau'_2,\dots,\tau'_{r'} \}$. For $i=1,2$, the central characters of $\pi_i$ and $\pi'_i$ are related by $\centralCharacter{\pi_i} = \centralCharacter{\pi'_i} \cdot \centralCharacter{\tau_1}$. Therefore, we have that $\pi'_1$ and $\pi'_2$ also have the same central character. By \Cref{prop:normalized-gamma-factor-full-multiplicativity}, we have that \eqref{eq:converse-theorem-condition-for-smaller-cuspidal-support} implies that for every $m \le \frac{n}{2}$ and every irreducible generic representation $\sigma$ of $\GL_m\left(\finiteField\right)$,  $$\normalizedGammaFactor{\pi'_1}{\sigma} = \normalizedGammaFactor{\pi'_2}{\sigma}.$$
	By induction $\pi'_1 \cong \pi'_2$, and therefore $\{\tau_2, \dots, \tau_r \} = \{ \tau'_2, \dots, \tau'_{r'} \}$. Hence, $\pi_1 \cong \pi_2$, as required.
	\end{proof}

\subsection{Special values of the Bessel function}

In this section, we use our results regarding multiplicativity of the Shahidi gamma factor, and its relation to the Jacquet--Piatetski-Shapiro--Shalika gamma factor in order to find an explicit formula for special values of the Bessel function of irreducible generic representations of $\GL_n\left(\finiteField\right)$. For two blocks, such a formula was given by Curtis and Shinoda in \cite[Lemma 3.5]{curtis2004zeta}. However, their proof uses Deligne--Lusztig theory, while our proof only uses Green's character values for irreducible cuspidal representation of $\GL_n\left(\finiteField\right)$, see \cite[Section 6]{Gelfand70} and \cite[Section 3.1]{Nien17}. We also provide a formula for a simple value consisting of three blocks. This generalizes a formula of Chang for irreducible generic representations of $\GL_3\left(\finiteField\right)$ \cite{chang1976decomposition}.

\subsubsection{Special value formula for two blocks}

Fix an algebraic closure $\overline{\finiteField}$ of $\finiteField$. For every positive integer $n$, let $\finiteField_n$ be the unique extension of degree $n$ in $\overline{\finiteField}$. Let $\FieldNorm{n} \colon \multiplicativegroup{\finiteField_n} \rightarrow \multiplicativegroup{\finiteField}$ and $\FieldTrace{n} \colon \finiteField_n \rightarrow \finiteField$ be the norm and the trace maps, respectively. Let $\characterGroup{\multiplicativegroup{\finiteField_n}}$ be the character group consisting of all multiplicative characters $\alpha \colon \multiplicativegroup{\finiteField_n} \rightarrow \multiplicativegroup{\cComplex}$.

It is known that irreducible cuspidal representations of $\GL_n\left(\finiteField\right)$ are in a bijection with Frobenius orbits of size $n$ of $\characterGroup{\multiplicativegroup{\finiteField_n}}$, that is, every irreducible cuspidal representation $\pi$ of $\GL_n\left(\finiteField\right)$ corresponds to a set of size $n$ of the form $\{ \alpha, \alpha^q, \dots, \alpha^{q^{n-1}} \}$, where $\alpha \in \characterGroup{\multiplicativegroup{\finiteField_n}}$.

We first recall Nien's result regarding the computation of the Jacquet--Piatetski-Shapiro--Shalika gamma factor $\rsGammaFactor{\pi}{\chi}$ where $\pi$ is an irreducible cuspidal representation of $\GL_n\left(\finiteField\right)$ and $\chi$ is a representation of $\GL_1\left(\finiteField\right)$, that is, $\chi \colon \multiplicativegroup{\finiteField} \rightarrow \multiplicativegroup{\cComplex}$ is a multiplicative character. Nien's result expresses $\rsGammaFactor{\pi}{\chi}$ as a Gauss sum.

\begin{proposition}[{\cite[Theorem 1.1]{Nien14}}]\label{prop:nien-computation-of-GLn-GL1-for-cuspidal}
	Let $\pi$ be an irreducible cuspidal representation of $\GL_n\left(\finiteField\right)$ associated with the Frobenius orbit $\{\alpha, \alpha^q, \dots, \alpha^{q^{n-1}} \}$, where $\alpha \in  \characterGroup{\multiplicativegroup{\finiteField_n}}$. Let $\chi \colon \multiplicativegroup{\finiteField} \rightarrow \multiplicativegroup{\cComplex}$ be a multiplicative character. Then
	$$ \rsGammaFactor{\pi}{\chi} = \left(-1\right)^{n+1} \chi( -1 )^{n+1} q^{-n+1} \sum_{\xi \in \multiplicativegroup{\finiteField_n}} \alpha^{-1}\left( \xi \right) \chi^{-1}( \FieldNorm{n}(\xi) ) \fieldCharacter( \FieldTrace{n}(\xi) ).$$ 
\end{proposition}

Nien's proof only uses Green's character formula for irreducible cuspidal representations, and does not use Deligne--Lusztig theory. 

We are ready to state our result regarding special two blocks values of the Bessel function.

\begin{theorem}\label{thm:bessel-function-special-value}
	Let $n>1$, and let $\pi$ be an irreducible generic representation of $\GL_n\left(\finiteField\right)$ with cuspidal support $\left\{\pi_1,\dots,\pi_r\right\}$, where for every $1 \le j \le r$,  $\pi_j$ is an irreducible cuspidal representation of $\GL_{n_j}\left(\finiteField\right)$ corresponding to the Frobenius orbit $\{\alpha_j, \alpha_j^q, \dots, \alpha_j^{q^{n_j - 1}}\}$, where $\alpha_j \in \characterGroup{\multiplicativegroup{\finiteField_{n_j}}}$ is a multiplicative character. Then for any $c \in \multiplicativegroup{\finiteField}$,
	$$ \repBesselFunction{\pi} \begin{pmatrix}
		& \identityMatrix{n - 1}\\
		c
	\end{pmatrix} = \left(-1\right)^{n + r} q^{-n + 1} \sum_{\substack{\xi_1 \in \multiplicativegroup{\finiteField_{n_1}}, \dots, \xi_r \in \multiplicativegroup{\finiteField_{n_r}} \\
\prod_{j=1}^r \FieldNorm{n_j}\left(\xi_j\right) = \left(-1\right)^{n-1} c^{-1}}} \prod_{j=1}^r \left(\alpha_j^{-1}\left(\xi_j\right)\fieldCharacter\left(\FieldTrace{n_j}\left(\xi_j\right)\right)\right).$$
\end{theorem}
\begin{proof}
	By \Cref{thm:explicit-formula-for-gamma-factors-in-terms-of-bessel-functions}, we have that $$ \gammaFactor{\pi}{\chi} = q^{n - 1} \sum_{x \in \multiplicativegroup{\finiteField}} \repBesselFunction{\pi}\begin{pmatrix}
		& \identityMatrix{n-1}\\
		x
	\end{pmatrix} \chi^{-1}\left(-x\right).$$
	Multiplying by $\chi \left(-c\right)$ and averaging over all $\chi \in \characterGroup{\multiplicativegroup{\finiteField}}$, and using the fact that a sum of a non-trivial character on a group is zero, we get
	\begin{equation}\label{eq:mellin-transform-of-gamma-factor}
		\frac{1}{\sizeof{\multiplicativegroup{\finiteField}}}\sum_{\chi \in \characterGroup{\multiplicativegroup{\finiteField}}} \gammaFactor{\pi}{\chi} \chi\left(-c\right) =  q^{n-1} \repBesselFunction{\pi}\begin{pmatrix}
			& \identityMatrix{n-1}\\
			c
		\end{pmatrix}.
	\end{equation}
	By \Cref{cor:full-parabolic-induction-multiplicativity-property}, we have that $$\gammaFactor{\pi}{\chi} = \prod_{j=1}^r \gammaFactor{\pi_j}{\chi}.$$
	By \Cref{cor:relation-between-intertwining-and-rankin-selberg} and \Cref{prop:nien-computation-of-GLn-GL1-for-cuspidal}, we have that $$ \gammaFactor{\pi_j}{\chi} =   \left(-1\right)^{n_j+1} \chi( -1 )^{n_j} \sum_{\xi \in \multiplicativegroup{\finiteField_{n_j}}} \alpha_j^{-1}\left( \xi \right) \chi( \FieldNorm{n_j}(\xi) ) \fieldCharacter( \FieldTrace{n_j}(\xi) ).$$
	Therefore, we get that $\gammaFactor{\pi}{\chi}$ is given by
	\begin{equation}\label{eq:gamma-factor-for-gln-times-gl1-as-product-of-gauss-sums}
		\left(-1\right)^{n+r} \sum_{\xi_1 \in \multiplicativegroup{\finiteField_{n_1}}, \dots, \xi_r \in \multiplicativegroup{\finiteField_{n_r}}} \left(\prod_{j=1}^r \alpha_j^{-1}\left( \xi_j \right) \fieldCharacter(  \FieldTrace{n_j}(\xi_j) ) \right) \chi\left( ( -1 )^{n} \prod_{j=1}^r \FieldNorm{n_j}\left(\xi_j\right) \right).
	\end{equation}
	Substituting the expression \eqref{eq:gamma-factor-for-gln-times-gl1-as-product-of-gauss-sums} for $\gammaFactor{\pi}{\chi}$ in \eqref{eq:mellin-transform-of-gamma-factor}, and using the fact that a sum of a non-trivial of character over a group is zero, we get the desired result.
\end{proof}
\begin{remark}
	The expression for $\gammaFactor{\pi}{\chi}$ in \eqref{eq:gamma-factor-for-gln-times-gl1-as-product-of-gauss-sums} is originally due to Kondo \cite{Kondo1963}. He computed it for the Godement--Jacquet gamma factor. One can show directly that the Godement--Jacquet gamma factor coincides with the Shahidi gamma factor for representations for which both factors are defined. Our proof, which is based on Nien's result and on multiplicativity of gamma factors, is different than the one given by Kondo. See also another proof in \cite[Chapter IV, Section 6, Example 4]{macdonald1998symmetric}.
\end{remark}
\begin{remark}
	In \cite{zelingher2022values}, a vast generalization of the method in the proof of \Cref{thm:bessel-function-special-value} is used in order to find formulas for $$\repBesselFunction{\pi} \begin{pmatrix}
		& \identityMatrix{n-m}\\		
		c \identityMatrix{m}
	\end{pmatrix}.$$ However, \cite{zelingher2022values} relies on the results of \cite{ye2021epsilon}, which in turn rely on the local Langlands correspondence. The proof given here does not rely on such results.
\end{remark}

\subsubsection{Special value formula for three blocks}
In this subsection, we use our results to prove a formula for special values of the Bessel function, for a simple value consisting of three blocks. This generalizes a formula given by Chang \cite{chang1976decomposition} for $\GL_3\left(\finiteField\right)$, generalized later by Shinoda and Tulunay \cite{shinoda2005representations} to $\GL_4\left(\finiteField\right)$. Our proof is different from Chang's proof, which is based on the Gelfand--Graev algebra.

We start with the following proposition.

\begin{proposition}\label{prop:fourier-transform-of-bessel-function}
	Let $\pi$ be an irreducible generic representation of $\GL_n\left(\finiteField\right)$. Then for any $c \in \multiplicativegroup{\finiteField}$, and any $g \in \GL_n\left(\finiteField\right)$, we have
	\begin{align*}
	&	\repBesselFunction{\pi}\left(g\right) \repBesselFunction{\pi}\begin{pmatrix}
			& \identityMatrix{n-1}\\
			c
		\end{pmatrix} \\
	=& q^{-\left(n-1\right)} \sum_{\transpose{x} = \left(x_1,\dots,x_{n-1}\right) \in \finiteField^{n-1}}  \fieldCharacter \left(-x_{n-1} \right) \repBesselFunction{\pi}\left(g \begin{pmatrix}
			\identityMatrix{n-1} & x\\
			& 1
		\end{pmatrix} \begin{pmatrix}
			& \identityMatrix{n-1}\\
			c
		\end{pmatrix}\right).
	\end{align*}
\end{proposition}

\begin{proof}
	Let $m=1$ and let $\sigma = \chi \colon \multiplicativegroup{\finiteField} \rightarrow \multiplicativegroup{\cComplex}$ be a multiplicative character. By \eqref{eq:gamma-intermediate-m-smaller-than-n}, we have
	\begin{equation*}
		\gammaFactor{\pi}{\chi} \repBesselFunction{\pi}  = \sum_{\substack{\transpose{x} \in \finiteField^{n-1} \\
				a \in \multiplicativegroup{\finiteField}}} \chi\left(-a^{-1}\right) \fieldCharacter\begin{pmatrix}
			I_{n-1} & -a^{-1} x\\
			& 1
		\end{pmatrix} \pi \begin{pmatrix}
			x &  I_{n-1}\\
			a
		\end{pmatrix} \repBesselFunction{\pi}.
	\end{equation*}
	We multiply by $\chi\left(-c\right)$ and average over $\chi \in \characterGroup{\multiplicativegroup{\finiteField}}$. Using the fact that a sum of a non-trivial character over a group is zero, and using \eqref{eq:mellin-transform-of-gamma-factor}, we get
	$$ 		q^{n-1} \repBesselFunction{\pi} \begin{pmatrix}
		& \identityMatrix{n-1}\\
		c
	\end{pmatrix} \repBesselFunction{\pi}  = \sum_{\transpose{x} = \left(x_1,\dots,x_{n-1}\right) \in \finiteField^{n-1}} \fieldCharacter \left(-c^{-1} x_{n-1} \right) \pi \begin{pmatrix}
		x &  I_{n-1}\\
		c
	\end{pmatrix} \repBesselFunction{\pi}. $$
	Using the decomposition $$\begin{pmatrix}
		x &  I_{n-1}\\
		c
	\end{pmatrix} = \begin{pmatrix}
		I_{n-1} & c^{-1} x\\
		& 1
	\end{pmatrix} \begin{pmatrix}
		& I_{n-1}\\
		c
	\end{pmatrix},$$ and changing the summation variable $x$ to $c \cdot x$, we get the desired result.
\end{proof}

\begin{theorem}\label{thm:three-blocks-special-bessel-value}Suppose $n \ge 3$. Then for any irreducible generic representation $\pi$ of $\GL_n\left(\finiteField\right)$ and any $c, c' \in \multiplicativegroup{\finiteField}$, we have
	\begin{align*}
		\repBesselFunction{\pi} \begin{pmatrix}
			& & - c'\\
			& \identityMatrix{n-2} & \\
			c
		\end{pmatrix}=
			&\sum_{s \in \multiplicativegroup{\finiteField}} \repBesselFunction{\pi} \begin{pmatrix}
		& \identityMatrix{n-1}\\
		s^{-1} c
	\end{pmatrix} \repBesselFunction{\pi} \begin{pmatrix}
		& s c' \\
		\identityMatrix{n-1}
	\end{pmatrix} \left(\fieldCharacter\left(s\right) - 1\right) + \frac{\delta_{c c',1}}{q^{n-2}},
\end{align*}
where $$ \delta_{c c',1} = \begin{dcases}
	1 & c c' = 1,\\
	0 & \text{otherwise.}
\end{dcases} $$
\end{theorem}
\begin{proof}
	We substitute $g = \begin{pmatrix}
		& c'\\
		\identityMatrix{n-1}
	\end{pmatrix}$
in \Cref{prop:fourier-transform-of-bessel-function} to get
$$ \repBesselFunction{\pi} \begin{pmatrix}
	& \identityMatrix{n-1}\\
	c
\end{pmatrix} \repBesselFunction{\pi} \begin{pmatrix}
& c'\\
\identityMatrix{n-1}
\end{pmatrix} = q^{-\left(n-1\right)} \sum_{\transpose{x} = \left(x_1,\dots,x_{n-1}\right) \in \finiteField^{n-1}}  \fieldCharacter \left(-x_{n-1} \right) \repBesselFunction{\pi}\begin{pmatrix}cc' & \\
cx & I_{n-1}
\end{pmatrix}.$$

If $x_{n-1} = 0$, then $\begin{pmatrix}cc' & \\
	cx & I_{n-1}
\end{pmatrix}$ lies in the mirabolic subgroup. By \cite[Lemma 2.14]{Nien14}, we have that the Bessel function is zero for elements in the mirabolic subgroup that do not lie in the upper unipotent subgroup $\unipotentSubgroup_n$. Therefore, we get that if $x_{n-1} = 0$, then $x = 0$ and $$ \fieldCharacter \left(x_{n-1} \right) \repBesselFunction{\pi}\begin{pmatrix}cc' & \\
cx & I_{n-1}
\end{pmatrix} = \delta_{c c',1}.$$

Suppose now that $x_{n-1} = t \ne 0$. Denote $\transpose{x'} = \left(x_1, \dots, x_{n-2}\right) \in \finiteField^{n-2}$. Then we have

$$ \begin{pmatrix}cc' & \\
	cx & I_{n-1}
\end{pmatrix} = \begin{pmatrix}
1 & 0 & t^{-1} c'\\
& \identityMatrix{n-2} & t^{-1} x'\\
 &  & 1
\end{pmatrix} 
\begin{pmatrix}
	& & -t^{-1} c'\\
	& \identityMatrix{n-2} & \\
	tc
\end{pmatrix}
\begin{pmatrix}
	1 & 0 & \left(tc\right)^{-1}\\
	& \identityMatrix{n-2} & -t^{-1} x'\\
	&  & 1
\end{pmatrix}.$$
Since we have $q^{n-2}$ elements in $\finiteField^{n-1}$ with $x_{n-1} = t$, we get that
\begin{equation*}
	\repBesselFunction{\pi} \begin{pmatrix}
		& \identityMatrix{n-1}\\
		c
	\end{pmatrix} \repBesselFunction{\pi} \begin{pmatrix}
		& c'\\
		\identityMatrix{n-1}
	\end{pmatrix} = \frac{\delta_{c c',1}}{q^{n-1}} + q^{-1} \sum_{t \in \multiplicativegroup{\finiteField}}  \fieldCharacter \left(-t \right) \repBesselFunction{\pi} \begin{pmatrix}
		& & -t^{-1} c'\\
		& \identityMatrix{n-2} & \\
		tc
	\end{pmatrix}.
\end{equation*}
We proceed as in \cite[Page 379]{chang1976decomposition} and \cite[Lemma 4.2]{shinoda2005representations}. We replace $c$ with $s^{-1} c$ and $c'$ with $s c'$, where $s \in \multiplicativegroup{\finiteField}$, to get
\begin{equation}\label{eq:product-of-2-bessel-scaled}
	 \repBesselFunction{\pi} \begin{pmatrix}
		& \identityMatrix{n-1}\\
		s^{-1} c
	\end{pmatrix} \repBesselFunction{\pi} \begin{pmatrix}
		& s c'\\
		\identityMatrix{n-1}
	\end{pmatrix} = \frac{\delta_{c c',1}}{q^{n-1}} + q^{-1} \sum_{t \in \multiplicativegroup{\finiteField}}  \fieldCharacter \left(-s t \right) \repBesselFunction{\pi} \begin{pmatrix}
		& & -t^{-1} c'\\
		& \identityMatrix{n-2} & \\
		tc
	\end{pmatrix}.
\end{equation}
Summing \eqref{eq:product-of-2-bessel-scaled} over $s \in \multiplicativegroup{\finiteField}$, we get
\begin{equation}\label{eq:first-convolution-of-2-bessel-scaled}
	\sum_{s \in \multiplicativegroup{\finiteField}} \repBesselFunction{\pi} \begin{pmatrix}
		& \identityMatrix{n-1}\\
		s^{-1} c
	\end{pmatrix} \repBesselFunction{\pi} \begin{pmatrix}
		& s c'\\
		\identityMatrix{n-1}
	\end{pmatrix} = \frac{q-1}{q^{n-1}} \delta_{c c',1} - q^{-1} \sum_{t \in \multiplicativegroup{\finiteField}} \repBesselFunction{\pi} \begin{pmatrix}
		& & -t^{-1} c'\\
		& \identityMatrix{n-2} & \\
		tc
	\end{pmatrix}.
\end{equation}
Multiplying \eqref{eq:product-of-2-bessel-scaled} by $\psi\left(s\right)$ and summing over $s \in \multiplicativegroup{\finiteField}$, we get
\begin{equation}\label{eq:convolution-of-2-block-bessel} 
\begin{split}
		&\sum_{s \in \multiplicativegroup{\finiteField}} \repBesselFunction{\pi} \begin{pmatrix}
		& \identityMatrix{n-1}\\
		s^{-1} c
	\end{pmatrix} \repBesselFunction{\pi} \begin{pmatrix}
		& s c' \\
		\identityMatrix{n-1}
	\end{pmatrix} \fieldCharacter\left(s\right)\\
	=& -\frac{\delta_{c c',1}}{q^{n-1}} + \frac{q-1}{q} \repBesselFunction{\pi} \begin{pmatrix}
		& & - c'\\
		& \identityMatrix{n-2} & \\
		c
	\end{pmatrix} - q^{-1}\sum_{1 \ne t \in \multiplicativegroup{\finiteField}} \repBesselFunction{\pi} \begin{pmatrix}
		& & -t^{-1} c'\\
		& \identityMatrix{n-2} & \\
		tc
	\end{pmatrix}.
\end{split}
\end{equation}
Subtracting \eqref{eq:first-convolution-of-2-bessel-scaled} from \eqref{eq:convolution-of-2-block-bessel}, we get the desired result.
\end{proof}

\begin{remark}
	Using the formulas in \Cref{thm:bessel-function-special-value} and its proof, one can show that if the cuspidal support of $\pi$ does not contain any irreducible representation of $\GL_1\left(\finiteField\right)$, then we have a simpler formula:
	\begin{equation}\label{eq:simpler-bessel-3-block-formula}
		\repBesselFunction{\pi} \begin{pmatrix}
			& & - c'\\
			& \identityMatrix{n-2} & \\
			c
		\end{pmatrix}=
		\sum_{s \in \multiplicativegroup{\finiteField}} \repBesselFunction{\pi} \begin{pmatrix}
			& \identityMatrix{n-1}\\
			s^{-1} c
		\end{pmatrix} \repBesselFunction{\pi} \begin{pmatrix}
			& s c' \\
			\identityMatrix{n-1}
		\end{pmatrix}\fieldCharacter\left(s\right).
	\end{equation}
	However, if the cuspidal support of $\pi$ contains irreducible representations of $\GL_1\left(\finiteField\right)$, this simpler formula does not hold.
\end{remark}
\begin{remark}
	Using the expression in \Cref{thm:bessel-function-special-value}, we have that the expression on the right hand side of \eqref{eq:simpler-bessel-3-block-formula} is an exponential sum that generalizes the Friedlander--Iwaniec character sum, see \cite[Theorem 7.3, formula (27) and Remark 7.4]{kowalski2014gaps}. The Friedlander--Iwaniec character sum played a role in Zhang's work on the twin prime conjecture \cite{zhang2014bounded}.
\end{remark}

\appendix

\section{Computation of $\rsGammaFactor{\pi}{\contragredient{\pi}}$ when $\pi$ is cuspidal}\label{appendix:computation-of-rankin-selberg-pi-pi-dual}

In this appendix, we compute the Jacquet--Piatetski-Shapiro--Shalika gamma factor ${\rsGammaFactor{\pi}{\sigma}}$ in the special case where $\pi$ and $\sigma$ are irreducible cuspidal representations of $\GL_n\left(\finiteField\right)$ and $\pi \cong \contragredient{\sigma}$. We will prove the following theorem.

\begin{theorem}\label{thm:rankin-selberg-gamma-factor-of-pi-times-pi-dual}
	Let $\pi$ be an irreducible cuspidal representation of $\GL_n\left(\finiteField\right)$. Then
	$$ \rsGammaFactor{\pi}{\contragredient{\pi}} = -1.$$
\end{theorem}
This was done in \cite[Corollary 4.3]{Ye18}. We provide another proof, since the proof in \cite{Ye18} relies on results of representations of $p$-adic groups.

For future purposes, we will prove the following general lemma. We will show that \Cref{thm:rankin-selberg-gamma-factor-of-pi-times-pi-dual} follows from it.

We denote by $\mirabolic_n \le \GL_n \left( \finiteField \right)$ the mirabolic subgroup.
\begin{lemma}\label{lem:gamma-factors-exceptional-case}
	Let $G$ be a finite group and let $H \le G$ be a subgroup. Suppose that $H$ is a semi-direct product of the form $H = N \rtimes \GL_n\left(\finiteField\right)$. Let $\Psi \colon H \rightarrow \multiplicativegroup{\cComplex}$ be a character which is trivial on $\GL_n\left(\finiteField\right)$. Let $\tau$ be an irreducible representation of $G$, such that
	\begin{enumerate}
		\item \label{item:hom-space-H-is-one-dimensional} $ \dim \Hom_{H}\left( \representationRes{H} \tau, \Psi \right) = 1$.
		\item $ \dim \Hom_{ N \rtimes \mirabolic_n }\left( \representationRes{ N \rtimes \mirabolic_n} \tau, \representationRes{N \rtimes \mirabolic_n } \Psi \right) = 1$.
		\item There exists a functional $\ell \in \Hom_{\unipotentSubgroup_n}\left(\representationRes{\unipotentSubgroup_n} \tau, \cComplex\right)$ and a vector $v_0 \in \tau$, such that
		$$ \sum_{p \in \unipotentSubgroup_n \backslash \mirabolic_n} \sum_{n \in N} \ell\left( \tau\left(n p\right) v_0 \right) \Psi^{-1}\left(n\right) = 1.$$
	\end{enumerate}
Then
$$ \sum_{g \in \unipotentSubgroup_n \backslash \GL_n\left(\finiteField\right)} \sum_{n \in N} \ell\left( \tau\left(n g\right) v_0 \right) \Psi^{-1}\left(n\right) \fieldCharacter\left(\bilinearPairing{e_n g}{e_1}\right) = -1.$$
\end{lemma}
\begin{remark}
	If $\multiplicativegroup{\finiteField} \le H$ lies in the center of $G$, then (\ref{item:hom-space-H-is-one-dimensional}) implies that the restriction of the central character of $\tau$ to $\multiplicativegroup{\finiteField} \le H$ is trivial.
\end{remark}

\begin{proof}
	Notice that we have a containment 
	$$ \Hom_{H}\left( \representationRes{H} \tau, \Psi \right) \subset \Hom_{ N \rtimes \mirabolic_n }\left( \representationRes{ N \rtimes \mirabolic_n} \tau, \representationRes{N \rtimes \mirabolic_n } \Psi \right).$$
	Since both spaces are one dimensional, we have that they are equal. Denote for $v \in \tau$, $$ L\left(v\right) = \sum_{p \in \unipotentSubgroup_n \backslash \mirabolic_n} \sum_{n \in N} \ell\left( \tau\left(n p\right) v \right) \Psi^{-1}\left(n\right).$$ Then $L \in \Hom_{ N \rtimes \mirabolic_n }\left( \representationRes{ N \rtimes \mirabolic_n} \tau, \representationRes{N \rtimes \mirabolic_n } \Psi \right)$  and $L \ne 0$ because $L\left(v_0\right) = 1$. Therefore, $L \in \Hom_{H}\left( \representationRes{H} \tau, \Psi \right)$, which implies that $L\left(\tau\left(g\right) v \right) = L\left(v\right)$ for any $v \in \tau$, and any $g \in \GL_n\left(\finiteField\right)$. 
	
	Denote $$ S = \sum_{g \in \unipotentSubgroup_n \backslash \GL_n\left(\finiteField\right)} \sum_{n \in N} \ell\left( \tau\left(n g\right) v_0 \right) \Psi^{-1}\left(n\right) \fieldCharacter\left(\bilinearPairing{e_n g}{e_1}\right).$$ We have
	$$ S = \sum_{g \in \mirabolic_n \backslash \GL_n\left(\finiteField\right)} L\left(\tau\left(g\right) v_0\right) \fieldCharacter\left(\bilinearPairing{e_n g}{e_1}\right) = \sum_{g \in \mirabolic_n \backslash \GL_n\left(\finiteField\right)} \fieldCharacter\left(\bilinearPairing{e_n g}{e_1}\right).$$
	We decompose this sum through the center of $\GL_n\left(\finiteField\right)$
	$$ S = \sum_{g \in \left(\multiplicativegroup{\finiteField} \cdot \mirabolic_n\right) \backslash \GL_n\left(\finiteField\right)} \sum_{a \in \multiplicativegroup{\finiteField}} \fieldCharacter\left(\bilinearPairing{e_n a g}{e_1}\right).$$
	We have that for $t \in \finiteField$, $$ \sum_{a \in \multiplicativegroup{\finiteField}} \fieldCharacter\left(at\right) = \begin{dcases}
		-1 & t \ne 0,\\
		q-1 & t = 0.
	\end{dcases}$$
	Therefore, we get \begin{align*}
		S =& \left(q-1\right)\sum_{\substack{g \in \left(\multiplicativegroup{\finiteField} \cdot \mirabolic_n\right) \backslash \GL_n\left(\finiteField\right)\\
				\bilinearPairing{e_n g}{e_1} = 0}} 1 - \sum_{\substack{g \in \left(\multiplicativegroup{\finiteField} \cdot \mirabolic_n\right) \backslash \GL_n\left(\finiteField\right)\\
				\bilinearPairing{e_n g}{e_1} \ne 0}} 1,
	\end{align*}
which we rewrite as
\begin{align*}
	S =& \sum_{\substack{g \in \mirabolic_n \backslash \GL_n\left(\finiteField\right)\\
			\bilinearPairing{e_n g}{e_1} = 0}} 1 - \frac{1}{\sizeof{\multiplicativegroup{\finiteField}}} \sum_{\substack{g \in \mirabolic_n \backslash \GL_n\left(\finiteField\right)\\
			\bilinearPairing{e_n g}{e_1} \ne 0}} 1.
\end{align*}

	Consider the right action of $\GL_n\left(\finiteField\right)$ on $\finiteField^n \setminus \left\{0\right\}$. This action is transitive. The stabilizer of $e_n$ is the mirabolic subgroup $\mirabolic_n$. Therefore, for $x = \left(x_1,\dots,x_n\right) \in \finiteField^n \setminus \left\{0\right\}$, we have that $$S_x = \sum_{\substack{g \in \mirabolic_n \backslash \GL_n\left(\finiteField\right)\\
			e_n g = x}} 1 = 1.$$
		This implies that
		\begin{align*}
	S = \sum_{\substack{x \in \finiteField^n \setminus \left\{0\right\} \\
			x_1 = 0}}  S_x - \frac{1}{\sizeof{\multiplicativegroup{\finiteField}}} \sum_{\substack{x \in \finiteField^n \setminus \left\{0\right\} \\
			x_1 \ne 0}}  S_x = \left(q^{n-1}-1\right) - q^{n-1} = -1,
\end{align*}
as required.
\end{proof}

We move to prove \Cref{thm:rankin-selberg-gamma-factor-of-pi-times-pi-dual}.

\begin{proof}
	We will use \Cref{lem:gamma-factors-exceptional-case} in the following setup. Let $G = \GL_n\left(\finiteField\right) \times \GL_n\left(\finiteField\right)$, and let $H = \GL_n\left(\finiteField\right)$ embedded diagonally. Let $N = \left\{\identityMatrix{n} \right\}$ and $\Psi = 1$. 
	
	Let $\pi$ be an irreducible cuspidal representation of $\GL_n \left(\finiteField\right)$, then by Schur's lemma, the space $$\Hom_{\GL_n\left(\finiteField\right)}\left(\pi \otimes \contragredient{\pi}, \cComplex\right)$$ is one-dimensional. Since $\pi$ is cuspidal, By \cite[Theorem 2.2]{Gelfand70}, the restriction of $\pi$ to the mirabolic subgroup $\mirabolic_n$ is irreducible. Therefore, by Schur's lemma the space $$\Hom_{\mirabolic_n}\left(\representationRes{\mirabolic_n}\pi \otimes \representationRes{\mirabolic_n}\contragredient{\pi}, \cComplex \right)$$ is also one-dimensional. 
	
	We take $\tau = \whittaker\left(\pi, \fieldCharacter\right) \otimes \whittaker\left(\contragredient{\pi}, \fieldCharacter^{-1}\right)$, and $\ell \colon \whittaker\left(\pi, \fieldCharacter\right) \otimes \whittaker\left(\contragredient{\pi}, \fieldCharacter^{-1}\right) \rightarrow \cComplex$ to be the functional defined on pure tensors by $$\ell \left(W \otimes W'\right) = W\left(\identityMatrix{n}\right) \cdot W'\left(\identityMatrix{n}\right).$$ 
	We have that $\ell \in \Hom_{\unipotentSubgroup_n}\left(\representationRes{\unipotentSubgroup_n} \tau, 1\right)$. Let $v_0 = \repBesselFunction{\pi} \otimes \grepBesselFunction{\contragredient{\pi}}{\fieldCharacter^{-1}}$. 
	
	Consider $$\sum_{p \in \unipotentSubgroup_n \backslash \mirabolic_n} \sum_{n \in N} \ell\left( \tau\left(n p\right) v_0 \right) \Psi^{-1}\left(n\right) = \sum_{p \in \unipotentSubgroup_n \backslash \mirabolic_n} \repBesselFunction{\pi}\left(p\right) \grepBesselFunction{\contragredient{\pi}}{\fieldCharacter^{-1}}\left(p\right).$$
	 By \cite[Lemma 2.14]{Nien14}, we have that if $\repBesselFunction{\pi}\left(p\right) \ne 0$ for $p \in \mirabolic_n$, then  $p \in \unipotentSubgroup_n$. Therefore,  $$\sum_{p \in \unipotentSubgroup_n \backslash \mirabolic_n} \repBesselFunction{\pi}\left(p\right) \grepBesselFunction{\contragredient{\pi}}{\fieldCharacter^{-1}}\left(p\right) = \sum_{p \in \unipotentSubgroup_n \backslash \unipotentSubgroup_n} \repBesselFunction{\pi}\left(p\right) \grepBesselFunction{\contragredient{\pi}}{\fieldCharacter^{-1}}\left(p\right) = 1.$$
	Thus, we showed that the required properties for \Cref{lem:gamma-factors-exceptional-case} are satisfied.
	
	Using \Cref{prop:rankin-selberg-gamma-factor-m-equals-n-in-terms-of-bessel-functions}, we have
	$$ \rsGammaFactor{\pi}{\contragredient{\pi}} = \sum_{g \in \unipotentSubgroup_n \backslash \GL_n\left(\finiteField\right)} \repBesselFunction{\pi}\left(g\right) \grepBesselFunction{\contragredient{\pi}}{\fieldCharacter^{-1}}\left(g\right) \fieldCharacter\left(\bilinearPairing{e_n g^{-1}}{e_1}\right).$$
	Replacing $g$ with $g^{-1}$ and using \Cref{prop:contragredient-complex-conjugate}, we have
	$$ \rsGammaFactor{\pi}{\contragredient{\pi}} = \sum_{g \in \unipotentSubgroup_n \backslash \GL_n\left(\finiteField\right)} \repBesselFunction{\pi}\left(g\right) \grepBesselFunction{\contragredient{\pi}}{\fieldCharacter^{-1}}\left(g\right) \fieldCharacter\left(\bilinearPairing{e_n g}{e_1}\right),$$
	and therefore by \Cref{lem:gamma-factors-exceptional-case}
	$$ \rsGammaFactor{\pi}{\contragredient{\pi}} = \sum_{g \in \unipotentSubgroup_n \backslash \GL_n\left(\finiteField\right)} \sum_{n \in N} \ell\left( \tau\left(n g\right) v \right) \Psi^{-1}\left(n\right) \fieldCharacter\left(\bilinearPairing{e_n g}{e_1}\right) = -1,$$
	as required.
\end{proof}

\bibliographystyle{abbrv}
\bibliography{references}
\end{document}